\documentclass[12pt]{amsart}
\usepackage{calc,amssymb,amsthm,amsmath}
\usepackage{alltt}
\RequirePackage[dvipsnames,usenames]{color}

\usepackage{mabliautoref}
\usepackage{colonequals}
\input{kmacros3.sty}
\usepackage{stmaryrd}
\usepackage{tikz}
\usetikzlibrary{matrix,arrows}

\usepackage{fullpage}
\usepackage{hyperref}
\usepackage{verbatim}
\usepackage{enumerate}

\newtheorem*{claim*}{Claim}
\theoremstyle{theorem}

\theoremstyle{remark}
\DeclareMathOperator{\length}{\ell}
\usepackage{graphicx}

\usepackage[all,cmtip]{xy}

\renewcommand{\O}{\mathcal{O}}
\renewcommand{\m}{\mathfrak{m}}
\DeclareMathOperator{\cokernel}{cokernel}

\renewcommand{\:}{\colon}

\DeclareMathOperator{\sdim}{sdim}
\DeclareMathOperator{\vol}{vol}

\renewcommand{\phi}{\varphi}

\renewcommand{\theta}{\vartheta}

\renewcommand{\epsilon}{\varepsilon}

\renewcommand{\to}[1][]{\xrightarrow{\ #1\ }}

\newcommand{\into}[1][]{\lhook \joinrel \xrightarrow{\ #1\ }}

\renewcommand{\theenumi}{\alph{enumi}}
\begin{document}

\title [$F$-signature of pairs]{$F$-signature of pairs and the asymptotic behavior of Frobenius splittings}
\author{Manuel Blickle, Karl Schwede, Kevin Tucker}

\address{ Institut f\"ur Mathematik\\ Johannes Gutenberg-Universit\"at Mainz\\55099 Mainz, Germany}
\email{blicklem@uni-mainz.de}
\address{Department of Mathematics\\ The Pennsylvania State University\\ University Park, PA, 16802, USA}
\email{schwede@math.psu.edu}
\address{Department of Mathematics\\ University of Utah\\ Salt Lake City, UT, 84112, USA}
\email{kevtuck@math.utah.edu}

\thanks{The first author was partially supported by a Heisenberg Fellowship and the SFB/TRR45}
\thanks{The second author was partially supported by an NSF postdoctoral fellowship and the NSF grant DMS \#1064485}
\thanks{The third author was partially supported by the NSF postdoctoral fellowship DMS \#1004344}

\subjclass[2010]{13A35, 13D40, 14B05, 13H10}
\keywords{$F$-signature, Cartier algebra, $F$-pure, $F$-regular, splitting prime, $F$-splitting ratio}
\maketitle

\begin{abstract}
We generalize $F$-signature to pairs $(R,\sD)$ where $\sD$ is a Cartier subalgebra on $R$ as defined by the first two authors.  In particular, we show the existence and positivity of the $F$-signature for any strongly $F$-regular pair.  In one application, we answer an open question of I. Aberbach and F. Enescu by showing that the $F$-splitting ratio of an arbitrary $F$-pure local ring is strictly positive.  Furthermore, we derive effective methods for computing the $F$-signature and the $F$-splitting ratio in the spirit of the work of R. Fedder.
\end{abstract}


\section{Introduction}
When working with rings or schemes in prime characteristic $p>0$, sections of the Frobenius endomorphism are called \mbox{\emph{$F$-splittings}}. When such a splitting exists the various iterates of Frobenius must split as well, and limiting constructions
often allow one to conclude numerous desirable algebraic and geometric
properties
\cite{HochsterRobertsRingsOfInvariants,BrionKumarFrobeniusSplitting}.
In this article, we investigate the following natural question
concerning the local asymptotic behavior of the number of splittings of large iterates of Frobenius.

\begin{question}[{\textit{cf.} \cite[Question 4.9]{AberbachEnescuStructureOfFPure}}]
\label{ques:motivatingquestion}
  If $R$ is a local ring of equal characteristic $p > 0$, how many splittings does the $e$-iterated Frobenius map $F^{e} \: R \to R$ have for $e \gg 0$?
\end{question}

We assume for simplicity in the introduction that $R$ is a complete
local domain with perfect residue field. In this case, the
module-finite inclusion $R \subseteq R^{1/p^{e}}$ into the corresponding ring of $p^{e}$-th roots of elements of $R$ is naturally identified with the $e$-iterated Frobenius.  One obtains a precise measure of the number of distinct splittings of $F^{e} \: R \to R$ by writing $R^{1/p^{e}} = R^{\oplus a_{e}} \oplus M_{e}$ as $R$-modules where $M_{e}$ has no free direct summands.  The number $a_{e}$ is independent of the corresponding direct sum decomposition and is called the \emph{$e$-th \mbox{$F$-splitting }number of $R$}.

If $d$ equals the dimension of $R$, a well-known result of E. Kunz gives that $a_{e} \leq p^{ed}$ for all $e > 0$ with equality if and only if $R$ is regular.  For arbitrary $R$, this observation motivated C. Huneke and G. Leuschke \cite{HunekeLeuschkeTwoTheoremsAboutMaximal}  (\cf
\cite{SmithVanDenBerghSimplicityOfDiff}) to consider the \emph{$F$-signature of $R$}
\[
s(R) = \lim_{e \to \infty} {a_e \over p^{ed}}
\]
which asymptotically compares the $F$-splitting numbers of $R$ to
those of a regular ring of the same dimension.  This limit, which has
only recently been shown to exist in full generality by the third
author \cite{TuckerFSignatureExists}, detects the severity of the
singularities of $R$.  In particular, I.~Aberbach and G.~Leuschke
\cite{AberbachLeuschke} have shown that $s(R) > 0$ if and only if $R$
is strongly $F$-regular.  In other words, if $R$ is strongly
$F$-regular, the answer to Question~\ref{ques:motivatingquestion} above is that the $F$-splitting numbers $a_{e}$ are (up to a positive constant) on the order of $p^{ed}$ for $e \gg 0$.  However, if $R$ is not strongly $F$-regular, then $s(R) = 0$ and so the $F$-splitting numbers $a_e$ grow at a rate strictly less than $p^{ed}$ for $e \gg 0$.

The aim of this article is to vastly generalize the theory of $F$-signature to so-called pairs.  For example, we shall give definitions for the $F$-signature of the ideal pairs $(R, \mathfrak{a}^{t})$ used in \cite{HaraYoshidaGeneralizationOfTightClosure} as well as the divisor pairs from \cite{HaraWatanabeFRegFPure} that appear throughout birational algebraic geometry.  While interesting in its own right, this theory also has a number of powerful applications in the classical setting.  In particular, it can be used to give an answer to Question~\ref{ques:motivatingquestion} in full generality.

In \cite{AberbachEnescuStructureOfFPure}, I. Aberbach and F. Enescu approached Question~\ref{ques:motivatingquestion} through the use of a naturally defined prime ideal $P$ called the \emph{$F$-splitting prime of $R$}.  In particular, they showed that the growth rate of the $F$-splitting numbers $a_{e}$ is bounded above by $p^{e\dim(R/P)}$, and they further proposed to study the limit (shown to exist in \cite{TuckerFSignatureExists})
\[
r_{F}(R) = \lim_{e \to \infty} \frac{a_{e}}{p^{e\dim(R/P)}}
\]
called the \emph{$F$-splitting ratio of $R$}.  One of the main applications of $F$-signature of pairs contained herein is the following result.

\medskip
\noindent
{\bfseries Corollary~\ref{cor:fsplitratio}.}
 If $R$ is $F$-pure (\textit{i.e.} has an $F$-splitting), then $r_{F}(R)>0$ is strictly positive.
\medskip

This result gives a complete answer to Question~\ref{ques:motivatingquestion}, as we have that the $F$-splitting numbers are either all zero or are (up to a positive constant) on the order of $p^{e\dim(R/P)}$ for $e \gg 0$ where $P$ is the $F$-splitting prime of $R$.  We remark that our result was certainly anticipated by I. Aberbach and F. Enescu (\textit{cf.} \cite[Question 4.9]{AberbachEnescuStructureOfFPure} and Remark~\ref{rmk:aeremark} below).

Let us briefly sketch the proof of the above result as a lead in to further discussion of \mbox{$F$-signature} of pairs.  In the notation above, one first observes that $R/P$ is strongly \mbox{$F$-regular}.  Thus, the $F$-splitting numbers of $R/P$ grow like $p^{e\dim(R/P)}$ for $e \gg 0$, suggesting a change of setting to $R/P$.
Indeed, every $R$-linear map $R^{1/p^e} \to R$ induces an $R/P$-linear map $(R/P)^{1/p^e} \to R/P$, so that the $F$-splittings of $R$ induce $F$-splittings of $R/P$.  However, the difficulty lies in that a large number of $F$-splittings of $R/P$ do not arise in this manner (see Section~\ref{sec:whitneys-umbrella} for an explicit example). Rather, one needs a way to enumerate only certain kinds of \mbox{$F$-splittings}, which is precisely the idea of $F$-signature of pairs.

More generally, one starts with specified collections $\sD_{e}$ of $R$-linear maps $R^{1/p^e} \to R$ for each $e \geq 0$.  Following \cite{SchwedeTestIdealsInNonQGor, BlickleTestIdealsViaAlgebras}, we require that these collections $\sD_{e}$ can be put together to form a non-commutative $\N$-graded ring $\sD$ under function composition (after taking the necessary $p$-th roots), called a \emph{Cartier subalgebra on $R$}.  In particular, this ensures that there is a well-defined notion of $F$-regularity for the pair $(R, \sD)$. In Section~\ref{sec:gener-f-sign}, roughly speaking, we define the $F$-splitting number $a_{e}^{\sD}$ to be the maximal number of $F$-splittings contained in $\sD_{e}$.  The fundamental results of this article, summarized below, completely characterize the asymptotic growth of the numbers $a_{e}^{\sD}$ for $e \gg 0$ and readily imply Corollary~\ref{cor:fsplitratio} above.

\begin{theorem*}
Suppose $\sD$ is a Cartier subalgebra on a $d$-dimensional equal characteristic $p>0$ complete local domain
 $R$ with perfect residue field.  If $\Gamma_{\sD} = \{ \, e \; | \; a_{e}^{\sD} \neq 0 \, \}$, then
the limit
\[
s(R,\sD) = \lim_{e \in \Gamma_{\sD} \to \infty}
\frac{a_{e}^{\sD}}{p^{ed}}
\]
exists (Theorem \ref{thm:existence}) and is called the \emph{$F$-signature} of $(R, \sD)$.  Furthermore, $s(R, \sD)>0$ is positive if and only if
$(R, \sD)$ is $F$-regular (Theorem~\ref{thm:positivity}).
\end{theorem*}

The existence statement in the theorem above generalizes \cite{TuckerFSignatureExists} using related methods; however, it is quite distinct in that we make no appeal to Hilbert-Kunz multiplicity.  The subsequent positivity statement is a substantial generalization of the main result of \cite{AberbachLeuschke}.

Section \ref{sec:applications} is devoted to various applications and
examples of the theory of $F$-signature of pairs.  In addition to the
$F$-splitting ratio mentioned above, we also describe effective
Fedder-type methods for computing $F$-signature. This is done via the
formalism of $F$-graded systems following
\cite{BlickleTestIdealsViaAlgebras}.  As before, the crucial point is
that Cartier subalgebras facilitate changing settings from one ring to
another.  In particular, when $R$ is presented as the quotient of a
regular local ring $S$, we have the ability to translate difficult questions on $R$ to computable statements on $S$ (see Theorem \ref{thm:FedderComputationForFGradedSystems}).
As an immediate consequence one obtains the following, which recovers and extends a result of Aberbach and Enescu \cite[Discussion after Theorem~4.2]{AberbachEnescuStructureOfFPure} (see also \cite[Proposition 3.1]{EnescuYaoLowerSemicontinuity}).

\renewcommand{\n}{\mathfrak{n}}
\renewcommand{\a}{\mathfrak{a}}

\begin{theorem*}
Let $S$ be an $n$-dimensional complete regular local ring of equal characteristic $p > 0$ with maximal ideal $\n$ and perfect residue field.

\smallskip
  \begin{itemize}
  \item[(i)] (Corollary~\ref{cor:fedcritforsplitnumbersandsignature}) \textnormal{\cite[Discussion after Theorem 4.2]{AberbachEnescuStructureOfFPure}} If $R = S/J$ for some non-zero ideal $J$ and $d = \dim(R)$, then
\[
s(R) =\displaystyle \lim_{e \to \infty} \frac{1}{p^{ed}}\length_{S}\left(  S \big/ \left(
    \n^{[p^{e}]}:(J^{[p^{e}]}:J)\right) \right) \, \, .
\]
\item[(ii)] (Corollary~ \ref{cor:attightclosuretest}) If $\a$ is a non-zero ideal and $t \in \R_{>0}$, then all ideals of $S$ are $\a^{t}$-tightly closed (see \cite{HaraYoshidaGeneralizationOfTightClosure}) if and only if
\[
s(S,\a^{t}) = \lim_{e \to \infty} \frac{1}{p^{ne}}\left( S \big/ \left( \n^{[p^{e}]} : \a^{\lceil t(p^{e}-1) \rceil}\right)\right) > 0 \, \, .
\]
  \end{itemize}
\end{theorem*}

Finally, we conclude this article by computing $F$-signature in a number of interesting examples.  In particular, we exhibit an $F$-pure local ring $R$ with $F$-splitting prime $P$ such that $r_F(R) \neq s(R/P)$ (see Sections \ref{sec:whitneys-umbrella} and \ref{sec:ramified-cover-cusp}).  We also give a toric formula for the $F$-signature of monomial ideal pairs (see Section~\ref{sec:monomial-ideals}).
\vskip 6pt
\noindent \emph{Acknowledgements:}  The authors would like to thank
Craig Huneke for valuable and inspiring conversations, as well as
Florian Enescu for his comments on a previous draft of this article.  We would also like to thank the referee for numerous helpful comments.  Finally, the authors worked on this paper while visiting the Johannes Gutenberg-Universit\"at Mainz during the summers of 2010 and 2011.  These visits were funded by the SFB/TRR45 \emph{Periods, moduli, and the arithmetic of algebraic varieties}.

\section{Preliminaries}
\label{sec:preliminaries}


Unless explicitly stated otherwise, all rings throughout this paper
are assumed to be commutative with
a unit,
Noetherian, and to have
prime characteristic $p > 0$.  A
local ring is denoted by the tuple $(R, \m)$ where $\m$ is the unique maximal
ideal of the ring $R$, and if needed we denote by
$k = R/\m$ is the corresponding residue field.  The Frobenius
or $p$-th power endomorphism $F \: R \to R$ is defined by $r \mapsto
r^{p}$ for all $r \in R$.  Similarly, for $e \in \N$, we have $F^{e}
\: R \to R$ given by $r \mapsto r^{p^{e}}$.

Let $M$ be an $R$-module.  For any $e \in \N$, viewing $M$ as an
$R$-module via restriction of scalars for $F^{e}$, yields an
$R$-module we  denote by
$F^{e}_{*}M$.  Thus, $F^{e}_{*} M$ agrees with $M$ as an
Abelian group, and if $m \in M$ we set $F^{e}_{*}m$ to be the
corresponding element of $F^{e}_{*}M$.  Furthermore, for $r \in R$ it
follows that $r
(F^{e}_{*}m) = F^{e}_{*}(r^{p^{e}}  m)$.  Note that $F^{e}_{*}R$
inherits the structure of a ring abstractly isomorphic to $R$, and $F^{e}_{*}M$ is naturally an
$F^{e}_{*}R$-module for any $R$-module $M$.  It is critical to observe that $F^e_*$ is in fact a functor, and so maps between modules can also be pushed forward.

We have that $F^{e}_{*}R$
is an $R$-algebra via the homomorphism of $R$-modules $F^{e} \: R \to
F^{e}_{*}R$ given by $r
\mapsto F^{e}_{*} r^{p^{e}}$ for $r \in R$, which is but another
perspective on the $e$-th iterate of Frobenius.
In case $R$ is reduced, we may
identify $F^{e}_{*}R$ with the $R$-module $R^{1/p^{e}}$ of $p^{e}$-th
roots of $R$ by associating  $F^{e}_{*}r$ and $r^{1/p^{e}}$; the $e$-iterated Frobenius homomorphism
now takes on the
guise of the natural inclusion $R \subseteq R^{1/p^{e}}$.  Each point
of view has certain advantages, and we will
switch between them as the situation warrants throughout.

\begin{definition}
  Suppose $(R, \m)$ is a local ring of characteristic $p > 0$.  We
  say $R$ is \mbox{\emph{$F$-finite}} if $F_{*}R$ is finitely generated as an
  $R$-module, from which it follows that $F^{e}_{*}R$ is finitely
  generated for all $e \in \N$.  In this case, we set $\alpha(R) = \log_{p}[k:k^{p}]$.
\end{definition}

\noindent
Note that any local ring which is essentially of finite type over a
perfect field is \mbox{$F$-finite}, as well as a complete local
ring with perfect (or even $F$-finite) residue field.  We shall primarily
restrict our attention to $F$-finite rings throughout this article.

Denote by $\length_{R}(M)$ the length of a finitely generated
Artinian $R$-module $M$.  If $R$ is \mbox{$F$-finite} and $e \in \N$, it is easy to
see that
\[
\length_{F^{e}_{*}R}(F^{e}_{*}M) = \length_{R}(M)   \qquad
\length_{R}(F^{e}_{*}M) = p^{e\alpha(R)}\length_{R}(M)
\]
by using that $F^{e}_{*}( \blank )$ is an exact functor and
$[(F^{e}_{*}k \simeq k^{1/p^{e}}):k] = p^{e\alpha(R)}$. If $I =
\langle x_{1}, \ldots, x_{t} \rangle$ is an ideal in $R$, then the
corresponding ideal $I^{[p^{e}]} = \langle x_{1}^{p^{e}}, \ldots,
x_{t}^{p^{e}} \rangle$ satisfies $IF^{e}_{*}R =
F^{e}_{*}(I^{[p^{e}]})$ and is independent of the choice of generators of $I$.  In particular, if $I$ is $\m$-primary, we have
\[
\length_{R}\left( (R/I) \tensor_{R} F^{e}_{*}R \right) = \length_{R}\left( F^{e}_{*}\left(R/I^{[p^{e}]} \right) \right) = p^{e\alpha(R)} \length_{R}(R/I^{[p^{e}]}) \, \, .
\]

The following
results of Kunz, also treated in the appendix to
\cite{MatsumuraCommutativeAlgebra}, show how the Frobenius
endomorphism can be used to detect regularity.

\begin{theorem} \cite{KunzCharacterizationsOfRegularLocalRings,KunzOnNoetherianRingsOfCharP}
\label{thm:kunz}  Let $(R, \m)$ be an $F$-finite Noetherian local ring of dimension $d$ and characteristic $p >
  0$. Then $R$ is excellent and $\alpha(R_{P})
    = \alpha(R_{Q}) + \dim(R_{Q} / PR_{Q})$ for any two prime ideals
    $P \subseteq Q$ of $R$. Furthermore, $R$ is regular if and only if $F^{e}_{*}R$ is a free $R$-module of rank
    $p^{e(d+\alpha(R))}$ for some $e \in \Z_{>0}$, in which case
    $F^{e}_{*}R$ is a free  $R$-module of rank
    $p^{e(d+\alpha(R))}$ for all $e \in \Z_{>0}$.
\end{theorem}

\noindent
More generally, as suggested by Theorem~\ref{thm:kunz} above, one expects the $R$-module structure of $F^{e}_{*}R$ to reflect upon the singularities of $R$.  This observation is an essential part of the underlying motivation for the theory developed in this article.

\subsection{The Cartier algebra}
\label{sec:pminuselinearmaps}

In this section, we recall a general framework for the investigation
of singularities in positive characteristic commutative algebra.  It
is in the context of this framework that we proceed to develop a
theory of \mbox{$F$-signature} of pairs.

\begin{definition}
Suppose $R$ is an $F$-finite ring with prime characteristic $p > 0$. A \mbox{\emph{$p^{-e}$-linear}}  \emph{map} between $R$-modules $M$ and $N$ is an additive map $\phi \: M \to N$ such that $\phi(r^{p^e} x) = r \phi(x)$ for all $x \in M$ and $r \in R$.  Equivalently, it is simply an $R$-linear map $F^e_*M \rightarrow N$, so that the set of $p^{-e}$-linear maps from $M$ to $N$ is given by $\Hom_R(F^e_*M,N)$.  If $R$ is reduced and $M$ is a submodule of the total quotient ring of $R$, this set is also identified with $\Hom_R(M^{1/p^e},N)$ via the $R$-module isomorphism $M^{1/p^e} \to F^e_*M$ mapping $a^{1/p^{e}} \mapsto F^{e}_{*}a$.
\end{definition}

For $e \in \Z_{\geq 0}$, let $\sC^{R}_{e} = \Hom_{R}(F^{e}_{*}R, R)$ denote the
set of all $p^{-e}$-linear maps.  Consider now the Abelian group
\[
\sC^{R} = \bigoplus_{e \geq 0} \sC^{R}_{e} = \bigoplus_{e \geq 0}
\Hom_{R}(F^{e}_{*}R, R)
\]
where we have $\sC_{0}^{R} = \Hom_{R}(R,R) = R$.  In fact, $\sC^{R}$
carries the structure of a non-commutative $\N$-graded ring where the multiplication of homogenous elements is given by composition of additive maps. In the above notation this means that, if $\phi_{1} \in \sC^{R}_{e_{1}}$ and $\phi_{2} \in \sC^{R}_{e_{2}}$ are homogeneous elements, then the multiplication in $\sC^{R}$ is defined by
\[
\phi_{1} \cdot \phi_{2} := \left( (\phi_{1} \circ F^{e_{1}}_{*}
  \phi_{2}) \: F^{(e_{1}+e_{2})}_{*}R \to R \right) \in
\sC^{R}_{(e_{1}+e_{2})} \; .
\]
It is immediately verified that $\phi_1 \cdot \phi_2$ indeed corresponds to the composition $\phi_1 \circ \phi_2: R \to R$ if both are viewed as additive maps on $R$. In particular, the $n$-fold product of $\phi \in \sC^{R}_{e}$ is denoted by $\phi^n \in \sC^{R}_{ne}$ and corresponds simply to the $n$-fold composition of the additive map $\phi$ with itself.
The ring 
 $\sC^{R}$ is called either the
\emph{total ring of $p^{-e}$-linear maps on $R$} or simply the
\emph{(total) Cartier algebra on $R$}.  

From the fact that there is a natural ring inclusion $R  = \sC^R_0 \into \sC^R$ one should \emph{not} be tempted to conclude that $\sC^R$ is an $R$-algebra in the classical sense: It is only an $\mathbb{F}_p$-algebra since $R$ is generally not central in $\sC^R$. This remark also applies to the following definition.

\begin{definition} \cite{SchwedeTestIdealsInNonQGor}
  A \emph{ring of $p^{-e}$-linear maps} on an $F$-finite
  ring $R$ or simply a \emph{Cartier
    subalgebra on $R$} is
  a graded $\mathbb{F}_{p}$-subalgebra $\sD = \oplus_{e \geq 0} \sD_{e}$ of
  $\sC^{R}$ satisfying $\sD_{0} = \sC_{0} = \Hom_{R}(R,R) = R$ and
  $\sD_{e} \neq 0$ for some $e \in \Z_{>0}$.
\end{definition}

In the context of this paper, this theory was introduced in a paper of the second author, \cite{SchwedeTestIdealsInNonQGor}.  Further refinements and generalizations to arbitrary modules were developed in a paper of the first author \cite{BlickleTestIdealsViaAlgebras}.  This theory also has roots in \cite{LyubeznikSmithCommutationOfTestIdealWithLocalization} where the Matlis dual of the complete Cartier algebra was studied.  For a brief survey of these rings, see \cite[Section 7]{SchwedeTuckerTestIdealSurvey}.

The advantage of using the formalism of algebras of $p^{-e}$-linear
maps is that they
allow one to treat uniformly several common settings, such as:

\begin{enumerate}
\item The study of pairs $(R, \Delta)$ where $R$ is normal and
  $\Delta$ is an effective $\R$-divisor on $\Spec R$.
\item The study of pairs $(R,\ba^{t})$ where $t \in \R_{\geq 0}$
  and $\ba \subseteq R$ is an
  ideal not contained in the union of the minimal primes of $R$.
\item Triples $(R, \Delta, \ba^{t})$ combining the pairs in (a) and
  (b) above.
\end{enumerate}

\noindent
See Section~\ref{sec:divisor-pairs} for the precise construction of
the Cartier subalgebras appropriate for each of the above
variants.

\begin{example}
Let $R$ be an $F$-finite ring and $\phi \in
\sC^{R}_{e} = \Hom_{R}(F^{e}_{*}R,R)$ a non-zero $p^{-e}$-linear map.
The Cartier subalgebra generated by $\sC^{R}_{0} = R$ and
$\phi \in \sC^{R}_{e}$, \textit{i.e.} the smallest subring of
$\sC^{R}$ containing both $R$ and $\phi$, will be denoted by $\sC^{\phi}$.
Since $(r \cdot \phi)(\blank) = \phi(F^{e}_{*}r^{p^{e}} \cdot \blank)$, it is
easy to see
\[
\sC^{\phi}_{e} = \{ \, \phi(F^{e}_{*}r \cdot \blank )\; | \; r \in R \,\}
\]
so that $\phi$ generates $\sC_{e}^{\phi}$ as an
$F^{e}_{*}R$-module.
\end{example}

\begin{example}
\label{ex:gorensteincartieralgebra}
\cite[Remark 4.4]{SchwedeTestIdealsInNonQGor} (\cf \cite[Corollary 3.9]{SchwedeFAdjunction} and \cite{LyubeznikSmithCommutationOfTestIdealWithLocalization})
If $R$ is a Gorenstein local ring with canonical module $\omega_{R}
\simeq R$, then we have that $\Hom_{R}(F^{e}_{*}R,R) \simeq
\Hom_{R}(F^{e}_{*}R, \omega_{R}) \simeq F^{e}_{*} \omega_{R} \simeq
F^{e}_{*}R$ can be generated as an $F^{e}_{*}R$-module by a single
homomorphism $\Phi \in \sC^{R}_{1}$ (in fact, one may take $\Phi$ to
be the canonical dual of Frobenius).  Furthermore, in this case, one can show
$\sC^{\Phi} = \sC^{R}$.
While here the (total) Cartier algebra is finitely generated (over $\sC_{0}^{R} = R$), in general $\sC^R$ need not be finitely generated (\textit{e.g.} see \cite{KatzmanANonFinitelyGeneratedAlgebra, AlvarezmontanerBoixZarzuelaAlgebrasOfStanleyReisnerRings}).
\end{example}

\subsection{$F$-singularities}
\label{sec:test-ideals}

\begin{definition} \cite{HochsterRobertsFrobeniusLocalCohomology, HochsterHunekeTightClosureAndStrongFRegularity, SchwedeTestIdealsInNonQGor}
  Suppose $\sD$ is a Cartier subalgebra on an $F$-finite local ring $R$.
  \begin{enumerate}
  \item
 We say that $(R,\sD)$ is \emph{(sharply) $F$-pure} if there
  is a surjective homomorphism $\phi \in \sD_{e}$ for some $e \in
  \Z_{>0}$. The ring $R$ is called $F$-pure when $(R,\sC^{R})$ is $F$-pure.
  \item
  We say that $(R,\sD)$ is \emph{(strongly) $F$-regular}
  if it satisfies the following property: for all $c \in R$ not
  contained in any minimal prime, there
  exists an $e \in \N$ and $\phi \in \sD_{e}$ such that
  $\phi(F^{e}_{*}c) = 1$.
The ring $R$ is called (strongly) $F$-regular when $(R, \sC^{R})$ is $F$-regular.
\end{enumerate}
\end{definition}

\begin{remark}
  The reader is hereby warned that, throughout this article, our terminology differs
  in some instances from that which has been used historically.  Explicitly, what we call $F$-pure has historically been called \emph{sharply $F$-pure} and what we call \emph{$F$-regular} has historically been called \emph{strongly $F$-regular}. Therefore, by adding the qualifiers
  \emph{sharply} and \emph{strongly} to $F$-pure and $F$-regular,
  respectively, this
  discrepancy (which, in any event, does not occur in a number of
  cases) is easily rectified.  Likewise, whenever the test ideal appears in this paper, we really mean the \emph{big test ideal}, which is also known as the \emph{non-finitistic test ideal}.
\end{remark}

\begin{proposition}
  Suppose $\sD$ is a Cartier subalgebra on an $F$-finite local ring $R$.
\begin{enumerate}
  \item If $(R,\sD)$ is $F$-pure (respectively, $F$-regular) and $\sD
    \subseteq \sD'$ for another Cartier subalgebra $\sD'$, then
    $(R,\sD')$ is also $F$-pure (respectively, $F$-regular).
    \item $(R,\sD)$ is $F$-pure (respectively,
      $F$-regular) if and only if there exists some $e \in \Z_{>0}$ and $0 \neq
      \phi \in
      \sD_{e}$ such that $(R,\phi)$ is $F$-pure (respectively,
      $F$-regular).
  \item If $(R,\sD)$ is $F$-pure, then $R$ is reduced and weakly
    normal.
  \item If $(R,\sD)$ is $F$-regular, then $R$ is a
    Cohen-Macaulay normal domain.
  \end{enumerate}
\end{proposition}

\begin{proof}
  Statement {(a)} follows immediately from the definitions
  above.  In particular, since $\sD \subseteq \sC^{R}$, then $(R,
  \sD)$ being $F$-pure or $F$-regular implies the same property for
  $R$, respectively.  As a result, {(c)} and {(d)}
  follow from well-known results in the theory of tight closure
  \cite{HunekeTightClosureBook}.  For {(b)}, see \cite{SchwedeSmithLogFanoVsGloballyFRegular, SchwedeTestIdealsInNonQGor}.
\end{proof}

Suppose $\sD$ is a Cartier subalgebra on an $F$-finite ring $R$.  Viewing the homogenous elements of $\sD$ as additive maps on $R$ endows $R$ with the structure of a $\sD$-module.  Since $\sD_{0} = \Hom_{R}(R,R) = R$, every $\sD$-submodule of $R$ is necessarily also an $R$-submodule.
\begin{definition}
\label{def:compatideal}
Suppose $\sD$ is a Cartier subalgebra on an $F$-finite ring $R$.  An ideal $J \subseteq R$ is said to be \emph{$\sD$-compatible} if it is a $\sD$-submodule of $R$, so that $R/J$ inherits the structure of a $\sD$-module from $R$.  Equivalently, for all $e \in \Z_{\geq 0}$ and $\phi \in \sD_{e}$, we have that $\phi(F^{e}_{*}J) \subseteq J$ and thus $\phi$ induces a $p^{-e}$-linear map $\phi_{J}$ on $R/J$ fitting into the commutative diagram
\[
\xymatrix{
F^{e}_{*}R \ar[r]^{\phi} \ar@{->>}[d] & R \ar@{->>}[d] \\
F^{e}_{*}R/J \ar[r]_{\phi_{J}} & R/J \; .
}
\]
In particular, setting $(\sD_{J})_{e} = \{ \, \phi_{J} \; | \; \phi \in \sD_{e} \, \}$ for all $e \in \Z_{\geq 0}$, there is an induced Cartier subalgebra $\sD_{J} = \oplus_{e\geq 0} (\sD_{J})_{e}$ on $R/J$.
\end{definition}

\begin{theorem}
\label{thm:testidealsexist}
\cite{HochsterHunekeTC1}\cite[Proposition 3.23]{SchwedeTestIdealsInNonQGor}
For any Cartier subalgebra $\sD$ on an \mbox{$F$-finite} local domain $R$, there is a unique
smallest non-zero $\sD$-compatible ideal $\tau(R, \sD)$.  The ideal
$\tau(R, \sD)$ is called the \emph{(big) test ideal} of $(R, \sD)$.  Furthermore, $\tau(R, \sD) = R$ if and only if $(R, \sD)$ is (strongly) $F$-regular.
\end{theorem}

\begin{proof}
We have omitted a proof as we will not make use of the techniques in what follows, and refer the reader to the references listed above. The key point is to show there is a single non-zero element $b \in R$ which is contained in every $\sD$-compatible ideal; the argument is essentially the same as that which shows test elements exist in $F$-finite reduced rings. See \cite{SchwedeTuckerTestIdealSurvey} for a recent survey of test ideals from the point of view taken in this article.
\end{proof}

\begin{proposition}
\label{prop:splittingprimesexist}
\cite{AberbachEnescuStructureOfFPure,SchwedeCentersOfFPurity}
  Suppose $\sD$ is a Cartier subalgebra on an $F$-finite local ring $(R,\m,k)$.  The ideal
  \begin{equation*}
P_{\sD} = \{ \; r \in R \; | \;  \phi(F^{e}_{*}r) \in \m \mbox{ for all } e > 0
\mbox{ and all } \phi \in \sD_{e} \; \}\label{eq:splittingprime}
\end{equation*}
is a proper ideal if and only if $(R, \sD)$ is $F$-pure, in which case
it is a prime ideal and the largest proper
$\sD$-compatible ideal.  The ideal $P_{\sD}$ is called the
  \emph{$F$-splitting prime of $(R,\sD)$}.
\end{proposition}

\begin{proof}
 It is clear that $(R, \sD)$ is $F$-pure if and only if $P_{\sD}$ is a
 proper ideal.  Assuming this is the case, let us show that $P_{\sD}$
is prime.  If $c_{1}, c_{2} \in R \setminus P_{\sD}$, then for $i =
1,2$ there
exists $e_{i}>0$ and $\phi_{i} \in \sD_{e_{i}}$ with
$\phi_{i}(F^{e_{i}}_{*}c_{i}) = 1$.  But then $\psi \in
\sD_{e_{1}+e_{2}}$ given by
\[
\psi(\blank) = (\phi_{1} \cdot
\phi_{2})(F^{e_{1}+e_{2}}_{*}c_{1}^{p^{e_{2}}-1} \cdot \blank)
\]
satisfies $\psi(F^{e_{1}+e_{2}}_{*}(c_{1}c_{2})) = 1$.  Thus,
$c_{1}c_{2} \in R \setminus P_{\sD}$, and we have that $P_{\sD}$ is prime.
To see that $P_{\sD}$ is $\sD$-compatible, suppose we have $x \in P_{\sD}$ and $\phi \in \sD_{e}$ for some $e \in \Z_{>0}$.  For all $e' \in \Z_{>0}$ and $\phi' \in \sD_{e'}$, we must have $(\phi' \cdot \phi)(F^{e+e'}_{*}x) = \phi'\left(F^{e'}_{*}(\phi(F^{e}_{*}x)\right) \in \m$ since $\phi' \cdot \phi \in \sD_{e+e'}$ as $\sD$ is a Cartier subalgebra of $\sC^{R}$.  Thus, it follows that $\phi(F^{e}_{*}x) \in P_{\sD}$ and so $P_{\sD}$ is $\sD$-compatible.  Since any proper $\sD$-compatible ideal of $R$ is automatically contained in $P_{\sD}$, we see that $P_{\sD}$ is in fact the largest proper $\sD$-compatible ideal  (\textit{cf.} \cite{SchwedeCentersOfFPurity}).
\end{proof}

\begin{lemma}[{\cf \cite[Main Theorem, part (v)]{SchwedeFAdjunction}}]
\label{lem.BijectionBetweenCompatibleIdeals}
   Suppose $\sD$ is a Cartier subalgebra on an $F$-finite local ring
 $(R,\m,k)$.  If $J \subseteq R$ is a $\sD$-compatible ideal and $\sD_{J}$
is the induced Cartier subalgebra on $R/J$ as in
Definition~\ref{def:compatideal}, then the operation
\[
I \mapsto I / J
\]
induces a bijection between the $\sD$-compatible ideals $I$ of $R$ containing $J$ and the \mbox{$\sD_J$-compatible} ideals of $R/J$.

In particular, $P_{\sD}/J = P_{\sD_J}$.  Therefore, $(R/J, \sD_J)$ is $F$-regular if and only if $J$ is the $F$-splitting prime.
\end{lemma}
\begin{proof}
Fix $\phi \in \sD_e$ with induced $\phi_J \in (\sD_J)_e$.  The result is immediate after checking that $\phi_J(F^e_* (I/J)) \subseteq I/J$ if and only if $\phi(F^e_* I) \subseteq I$.  For the second statement, use the maximality of $P_{\sD}$ from Proposition \ref{prop:splittingprimesexist}.  The last statement follows from the final assertion in Theorem \ref{thm:testidealsexist}.
\end{proof}

\begin{remark}
The test ideal is the positive characteristic analog of the multiplier ideal from higher dimensional algebraic geometry.  Likewise, the subscheme defined by the splitting prime is analogous to the minimal LC-center.  For additional discussion, see \cite{SchwedeTuckerTestIdealSurvey}.
\end{remark}

\section{$F$-signature}
\label{sec:gener-f-sign}

The $F$-signature of a local ring $R$ is a numerical invariant which, roughly speaking, asymptotically compares the number of splittings of Frobenius on $R$ to the number of splittings one expects from a regular local ring with the same dimension.
Formally introduced by C.~Huneke and G.~Leuschke in
\cite{HunekeLeuschkeTwoTheoremsAboutMaximal} (\cf
\cite{SmithVanDenBerghSimplicityOfDiff}) the existence of the
$F$-signature was only recently shown in full generality by the third
author in \cite{TuckerFSignatureExists}.  In the subsequent sections,
we present a vast generalization of this invariant which incorporates
the additional data of a Cartier subalgebra $\sD$ on $R$.   Following
the initial definitions, we present two deep results concerning the existence (Theorem~\ref{thm:existence}) and positivity (Theorem~\ref{thm:positivity}) of the $F$-signature in this context.

\subsection{$F$-splitting numbers}
\label{sec:splitting-numbers}


\begin{definition}
   Suppose $\sD$ is a Cartier subalgebra on an $F$-finite local ring $R$ and $e \in \Z_{> 0}$.
Within a fixed direct sum decomposition $F^{e}_{*}R \simeq \bigoplus_{i} M_{i}$
as an $R$-module, the summand $M_j$ is said to be a
  \emph{$\sD$-summand} if $M_j \simeq R$ and the
  associated ($R$-linear) projection homomorphism $F^{e}_{*}R \to M_j \simeq R$
  belongs to $\sD_{e}$. Since $\sD_0= \Hom_{R}(R,R) = R$, this characterization is independent of the chosen isomorphism $M_j \simeq R$.

 The \emph{$e$-th
    $F$-splitting number of $(R, \sD)$} is the maximal number
  $a_{e}^{\sD}$ of
  $\sD$-summands appearing in the various direct sum decompositions of
  $F^{e}_{*}R$ as an $R$-module.
\end{definition}

\begin{remark}
Any direct sum decomposition $F^{e}_{*}R \simeq R^{\oplus a} \oplus M$
where the factors $R^{\oplus a}$ of $F^e_* R$ are $\sD$-summands may be further
refined to ensure that $M$ has no $\sD$-summands, at which point it will
follow from Proposition~\ref{prop:charofsplittingnumbers} below that $a
= a_{e}^{\sD}$.  However, the number of \mbox{$\sD$-summands} in arbitrary direct sum
  decompositions of $F^{e}_{*}R$ into indecomposable \mbox{$R$-modules} need
  not equal $a_{e}^{\sD}$ in general.
\end{remark}

 Our next goal is to build towards several alternate characterizations of
the $F$-splitting numbers in Proposition~\ref{prop:charofsplittingnumbers}.

\begin{definition}
    Suppose $\sD$ is a Cartier subalgebra on an $F$-finite local ring $R$ with maximal ideal $\m$. For $e \in \Z_{>0}$, we consider an ideal
\[
I_{e}^{\sD} = \{ \, r \in R \; | \; \phi(F^{e}_{*}r) \in \m \mbox{ for
  all } \phi \in \sD_{e} \, \}
\]
and a submodule of $\sD_e$
\[
\sD_{e}^{\textnormal{ns}} = \{ \, \phi \in \sD_{e} \; | \; \phi \mbox{ is not surjective,
  \textit{i.e.} } \phi(F^{e}_{*}R) \subseteq \m \, \} \; .
\]
In other words, $I_{e}^{\sD}$ can be described as those elements of $R$ whose
$p^{e}$-th roots cannot be made to generate a $\sD$-summand of
$F^{e}_{*}R$.  Similarly, $\sD_{e}^{\textnormal{ns}}$ is the set of
maps in $\sD_{e}$ which cannot serve as a projection homomorphism onto
a $\sD$-summand of $F^{e}_{*}R$.
\end{definition}

Before proceeding, let us first record some elementary properties of the ideals $I_{e}^{\sD}$.

\begin{lemma}
  \label{lem:propsofIe}
 Suppose $\sD$ is a Cartier subalgebra on an $F$-finite local ring $(R,\m)$.  If $\phi \in \sD_{e}$ for some $e \in \Z_{>0}$, then for all $e' \in \Z_{>0}$ we have $\m^{[p^{e'}]} \subseteq I_{e'}^{\sD}$ and $\phi(F^{e}_{*}I_{e+e'}^{\sD}) \subseteq I_{e'}^{\sD}$. In particular, if $\phi$ is surjective then $I_{e + e'}^{\sD} \subseteq I_{e'}^{\sD}$.
\end{lemma}

\begin{proof}
If $x \in \m$ and $\phi' \in \sD_{e'}$ for some $e' \in \Z_{>0}$, we have $\phi'(F^{e'}_{*}x^{p^{e'}}) = \phi'(x \cdot F^{e'}_{*}1) = x \phi'(F^{e'}_{*}1) \in \m$.  Thus, we conclude $\m^{[p^{e'}]} \subseteq I_{e'}^{\sD}$.   If now $y \in I_{e+e'}^{\sD}$, we must have
\[
(\phi' \cdot \phi)(F^{e+e'}_{*}y) = \phi'\left( F^{e'}_{*}\left( \phi(F^{e}_{*}y) \right) \right) \in \m
 \]
by the definition of $I_{e+e'}^{\sD}$ since $\phi' \cdot \phi \in \sD_{e+e'}$ as $\sD$ is closed under multiplication inside $\sC_{R}$.  It follows that $\phi(F^{e}_{*}I^{\sD}_{e+e'}) \subseteq I_{e'}^{\sD}$.
If additionally $\phi$ is surjective, we again have $(\phi \cdot \phi')(F^{e+e'}_{*}y) = \phi\left( F^{e}_{*}\left( \phi'(F^{e'}_{*}y) \right) \right) \subseteq \m$ and so necessarily $\phi'(F^{e'}_{*}y) \in \m$.  Therefore $y \in I_{e'}^{\sD}$ and so $I_{e+e'}^{\sD} \subseteq I_{e'}^{\sD}$ as desired.
\end{proof}

\begin{proposition}
\label{prop:charofsplittingnumbers}
  If $\sD$ is a Cartier subalgebra on an $F$-finite local ring
  $(R,\m)$, then for all $e \in \Z_{> 0}$ we have
\begin{enumerate}
\item
\parbox{\textwidth}{\centering $a_{e}^{\sD} =
\length_{R}(\sD_{e}/\sD_{e}^{\textnormal{ns}}) = \length_{R}(F^{e}_{*}(R/I_{e}^{\sD})) = p^{e \alpha(R)}
\length_{R}( R/I_{e}^{\sD} ) \; .$}
\end{enumerate}
Furthermore, the $F$-splitting number $a_{e}^{\sD}$ can alternatively
be described as either
\begin{enumerate}
\setcounter{enumi}{1}
\item the maximal
possible number
of direct summands isomorphic to $R$ in any direct sum decomposition of $\sC_{e}^{R}$
as an $R$-module which are contained inside $\sD_{e}$, or
\item the maximal
possible number of direct summands isomorphic to $R$ in any direct sum
decomposition of $\sD_{e}$ as an $R$-module whose associated
projection homomorphism $\sD_{e} \to R$ can be extended to a
homomorphism $\sC_{e}^{R} \to R$.
\end{enumerate}
\end{proposition}

Roughly speaking, we interpret the equality $a_{e}^{\sD} =
\length_{R}(\sD_{e}/\sD_{e}^{\textnormal{ns}})$ above to say that the $e$-th splitting
number of $(R, \sD)$ counts the number of splittings of the $e$-th
iterate of Frobenius that are contained in $\sD$.  In a sense, this
gives intuition into why one would expect a Cartier subalgebra $\sD$ with
smaller splitting numbers to correspond to more severe singularities.
The proof of Proposition~\ref{prop:charofsplittingnumbers}
is immediate
from the following more general lemma.

\begin{lemma}
\label{lemma:aboutsplittingnumbers}
  Let $(R, \m)$ be a local ring and $M$ a finitely generated reflexive $R$-module, and suppose $D \subset \Hom_{R}(M,R)$ is an $R$-submodule.
Then we have
\[
\length_{R}\left( M / \langle m \in M \, | \, \phi(m) \in \m \mbox{ for all } \phi \in D \rangle \right)
= \length_{R}\left( D / \langle \phi \in D \, | \, \phi(M) \subseteq \m \rangle \right) \, .
\]
Furthermore, this quantity can be expressed in terms of direct sum decompositions as follows:
\begin{enumerate}
\item
The maximal number of direct summands of M isomorphic to $R$ whose associated projection homomorphisms $M \to R$ belong to $D$.
\item
The maximal number of direct summands of $\Hom_{R}(M,R)$ isomorphic to $R$ which are contained in $D$.
\item
The maximal number of direct summands of $D$ isomorphic to $R$ whose associated projection homomorphisms $D \to R$ extend to a homomorphism $\Hom_{R}(M,R) \to R$.
\end{enumerate}
\end{lemma}

\begin{proof}
  Let  $M = R^{\oplus a} \oplus N$ be any direct sum decomposition which is maximal with respect to having all of the corresponding projections $\phi_{1}, \ldots, \phi_{a} \in \Hom_{R}(M,R)$ belong to $D$. In other words, $N$ cannot be further decomposed as $N = R \oplus N'$ in such a way that the additional projection $M \to R$ belongs to $D$.  There is a dual direct sum decomposition
$\Hom_{R}(M,R) = R^{\oplus a} \oplus \Hom_{R}(N,R)$ where $\phi_{1}, \ldots, \phi_{a}$ form a free basis for $R^{\oplus a} \subseteq \Hom_{R}(M,R)$. Let $e_{1}, \ldots, e_{a} \in R$ be the corresponding dual free basis for $R^{\oplus a} \subseteq M$ in the initial direct sum decomposition of $M$. Since $\phi_{1}, \ldots, \phi_{a} \in D$, we also have an induced direct sum decomposition $D = R^{\oplus a} \oplus (\Hom_{R}(N,R) \cap D)$.

If $\phi \in D$, we claim $\phi(N) \subseteq \m$.  Indeed, supposing otherwise there exists $n \in N$ with $\phi(n) = 1$.  Replacing $\phi$ by $\phi - \sum_{i=1}^{a} \phi(e_{i})\phi_{i}$, we may assume $\phi \in \Hom_{R}(N,R)$ which we view as a subset of $\Hom_{R}(M,R)$ via the projection $M = R^{\oplus a}  \oplus N \to N$.  Thus, the inclusion $Rn \subseteq N$ is split and induces a decomposition $N = R \oplus N'$ in such a way that the projection homomorphism $M \to R$ is $\phi$ itself. Since $\phi \in D$, this is a contradiction and we conclude $\phi(N) \subseteq \m$ for all $\phi \in D$.

Thus, with respect to the direct sum decompositions of $M$ and $D$ above, we have
\[
 \langle m \in M \, | \, \phi(m) \in \m \mbox{ for all } \phi \in D \rangle = \m^{\oplus a} \oplus N \subseteq R^{\oplus a} \oplus N = M
\]
\[
\langle \phi \in D \, | \, \phi(M) \subseteq \m \rangle = \m^{\oplus a} \oplus (\Hom_{R}(N,R) \cap D) \subseteq R^{\oplus a} \oplus (\Hom_{R}(N,R) \cap D) = D\, .
\]
In particular, both quotients are isomorphic to $k^{\oplus a}$ so that we have
\[
\length_{R}\left( M / \langle m \in M \, | \, \phi(m) \in \m \mbox{ for all } \phi \in D \rangle \right)
= \length_{R}\left( D / \langle \phi \in D \, | \, \phi(M) \subseteq \m \rangle \right)  = a \, .
\]
This shows the desired equality, as well as the equivalence with (a).

To show equivalence with (b), let $\Hom_{R}(M,R) = R^{\oplus a'} \oplus D'$ be a direct sum decomposition such that $R^{\oplus a'} \subseteq D$ and  $D'$ has no direct summand which is isomorphic to $R$ and contained in $D$. If $\phi_{1}', \ldots, \phi_{a'}'$ form a free basis for $R^{\oplus a'} \subseteq D \subseteq \Hom_{R}(M,R)$, then applying $\Hom_{R}(\blank,R)$ and using that $M$ is reflexive gives a direct sum decomposition $M = R^{\oplus a'} \oplus \Hom_{R}(D',R)$ where the projection homomorphisms onto the components of $R^{\oplus a'}$ are $\phi_{1}', \ldots, \phi_{a'}'$ and $\Hom_{R}(D',R)$ has no direct summand isomorphic to $R$ inducing a projection map $M \to R$ inside of $D$.  Thus, we must have $a' = a$ as desired.  The equivalence of (b) and (c) is straightforward.
\end{proof}

\begin{corollary}
  \label{cor:completionhassamenumbers}
 Suppose $\sD$ is a Cartier subalgebra on an $F$-finite local ring
  $(R,\m)$ with $\m$-adic completion $\hat{R}$. Let $\hat\sD$ denote the induced Cartier subalgebra on $\hat R$, \textit{i.e.} $\hat\sD_{e}$ is the image of $\sD_{e}$ under the natural map $\sC_{e}^{R} \to \hat R \tensor_{R} \sC_{e}^{R} = \sC_{e}^{\hat R}$ for all $e \in \Z_{>0}$.  Then $a_{e}^{\sD} = a_{e}^{\hat \sD}$.
\end{corollary}

\begin{proof}
Since $\hat\sD_{e}$ is the image of $\sD_{e}$, it follows that $I_{e}^{\hat\sD} \cap R = I_{e}^{\sD}$.  As $\m^{[p^{e}]} \subseteq I_{e}^{\sD}$, we conclude $I_{e}^{\sD} \hat R = I_{e}^{\hat\sD}$ and so the statement follows from Proposition~\ref{prop:charofsplittingnumbers}~(a).
\end{proof}

\subsection{Existence}
\label{sec:gener-f-sign-1}

\begin{definition}
   Suppose $\sD$ is a Cartier subalgebra on an $F$-finite local ring
 $(R,\m)$.  The \emph{semigroup of $\sD$}
  is the set $\Gamma_{\sD} = \{ \, e \in \Z_{\geq 0} \; | \; a_{e}^{\sD} \neq 0
  \, \}$, and the \emph{semigroup index of $\sD$} is $n_{\sD} =
  \gcd( \Gamma_{\sD} )$.
\end{definition}

\begin{lemma}
   Suppose $\sD$ is a Cartier subalgebra on an $F$-finite local ring
 $(R,\m)$.
 Then $\Gamma_{\sD}$ is a subsemigroup of $(\Z_{\geq 0}, +)$, and
 $n_{\sD} =\# (\Z / \Z\Gamma_{\sD} )$ is the index in $\Z$ of the
 subgroup generated by $\Gamma_{\sD}$.
Furthermore, if $(R,\sD)$ is $F$-regular, then $n_{\sD} = \gcd ( \{ \, e \in
\Z_{\geq 0} \; | \; \sD_{e} \neq 0 \, \})$.
\end{lemma}

\begin{proof}
Since $\sD$ is closed under multiplication and the composition of surjective maps is surjective, we have that $\Gamma_{\sD}$ is a subsemigroup of $(\Z_{\geq 0}, +)$ and thus also $n_{\sD} = \# (\Z / \Z\Gamma_{\sD})$.  For the last statement, assume now that $(R, \sD)$ is $F$-regular.  It suffices to show that $n_{\sD}$ divides $e$ whenever $\sD_e \neq 0$. Let $0 \neq \psi \in \sD_e$ with $\psi(F^e_*a)=b \neq 0$ for some $a,b \in R$. Using \mbox{$F$-regularity}, there is a $\phi \in \sD_{e'}$ with $\phi(F^{e'}_*b)=1$ for some $e'>0$. In particular, $a_{e'}^{\sD} \neq 0$ so that $n_{\sD}$ divides $e'$. Since $(\phi \cdot \psi )(F^{e+e'}_*a)=1$ it follows that $a_{e+e'}^\sD \neq 0$ and hence $n_\sD$ divides $e+e'$ as well. It now follows that $n_{\sD}$ divides $e$ as required.
\end{proof}

\begin{remark}
If $\sD$ is a Cartier subalgebra on an $F$-finite local domain
 $(R,\m)$, then $\{ \, e \in \Z_{\geq 0} \; | \; \sD_{e} \neq 0 \, \}$ is an alternative semigroup naturally associated to $\sD$.  However, this semigroup is not preserved when passing to quotients by $\sD$-compatible ideals (\textit{cf.} Lemma~\ref{lem:passtoquotients}).
Furthermore, note that it is straightforward to construct examples where $(R, \sD)$ is $F$-pure but $n_{\sD} \neq \gcd ( \{ \, e \in
\Z_{\geq 0} \; | \; \sD_{e} \neq 0 \, \})$.
\end{remark}

\begin{theorem}
  \label{thm:existence}
Suppose $\sD$ is a Cartier subalgebra on a $d$-dimensional $F$-finite local ring
 $(R,\m)$. Then the limit
\[
s(R,\sD) = \lim_{e \in \Gamma_{\sD} \to \infty}
\frac{a_{e}^{\sD}}{p^{e(d+\alpha(R))}} = \lim_{m \to \infty}
\frac{a_{mn_{\sD}}^{\sD}}{p^{mn_{\sD}(d+\alpha(R))}}
\]
exists and is called the \emph{$F$-signature} of $(R, \sD)$.
\end{theorem}

\begin{remark}
  Note that $a_{e}^{\sD} \leq \rank(F^{e}_{*}R) = p^{e(d+\alpha(R))}$ for all $e >0$, so that $s(R, \sD) \leq 1$.  See also \cite[Corollary 16]{HunekeLeuschkeTwoTheoremsAboutMaximal} and \cite[Theorem 3.1]{YaoObservationsAboutTheFSignature}.
\end{remark}

The proof of Theorem~\ref{thm:existence} is based on a series of lemmas,
the first of which also features prominently in subsequent sections.

\begin{lemma}
  \label{lem:directsumstuff}
Suppose $(R,\m)$ is a $d$-dimensional $F$-finite local domain
 and $\phi \in \sC_{n} =
 \Hom_{R}(R^{1/p^{n}},R)$ is a non-zero $p^{-n}$-linear map for some $n
 \in \Z_{>0}$.
 \begin{enumerate}
 \item There exists $0 \neq h \in R$ and an inclusion of $R$-modules
   $\mu \: R^{1/p^{n}} \to R^{\oplus p^{n(d+\alpha(R))}}$ such that $\cokernel(\mu)$ is annihilated by
   $h$ and each of the $p^{n(d+\alpha(R))}$ ($R$-linear) component functions
   $R^{1/p^{n}} \to R$ of $\mu$ are $R^{1/p^{n}}$-multiples of $\phi$.
 \item Given $h$ and $\mu$ as in (a), there exists for each $m \in \Z_{>0}$ an inclusion of $R$-modules $\mu_{m} \: R^{1/p^{nm}} \to R^{\oplus p^{nm(d+\alpha(R))}}$ determined by $\mu$ such that $\cokernel(\mu_{m})$ is annihilated
 by $g = h^{2}$ and each of the $p^{nm(d+\alpha(R))}$
   ($R$-linear) component functions $R^{1/p^{nm}} \to R$ of $\mu_{m}$
   are $R^{1/p^{nm}}$-multiples of $\phi^{m}$.
 \item  For $g$ as in (b),
if $\psi \in \sC_{nm}$ for some $m \in \Z_{>0}$, then $g \cdot \psi$
is a $R^{1/p^{nm}}$-multiple of $\phi^{m}$.  In other words, there
exists $r \in R$ with $g \cdot \psi(\blank) = \phi^{m}(F^{nm}_{*}r
\cdot \blank)$.
 \end{enumerate}
\end{lemma}

\begin{proof}
Let $\delta = d + \alpha(R)$. The field of fractions $K$ of $R$ has the property that $K^{1/p^e}$ is free over $K$. Furthermore the map $\phi$ is a nonzero element, and hence a generator, of the one-dimensional $K^{1/p^e}$-vector space $\Hom_K(K^{1/p^e},K)$. This two properties hold also on an sufficiently small open subset of $\Spec R$, hence we can find $0 \neq c \in R$ such that $(R^{1/p^{n}})_{c} = (R_{c})^{1/p^{n}}$ is a free $R_{c}$-module and $\phi_{c}$ generates $\Hom_{R_{c}}((R_{c})^{1/p^{n}}, R_{c})$ as an \mbox{$(R_{c})^{1/p^{n}}$-module}.  In particular, for any $\psi \in \sC_{n}$, there is a sufficiently large $N_{\psi}$ (depending on $\psi$) such that $c^{(N_{\psi})}\cdot \psi$ is an $R^{1/p^{n}}$-multiple of $\phi$. Since $R^{1/p^{n}}$ is a torsion-free $R$-module of rank $p^{n\delta}$, one can further assume that we have an inclusion of $R$-modules $\nu \: R^{1/p^{n}} \to R^{\oplus p^{n\delta}}$ which becomes an isomorphism after localizing at $c$.  Thus, setting
\[
\mu \: R^{1/p^{n}} \to[\nu] R^{\oplus p^{n\delta}} \to[\cdot c^{M}] R^{\oplus p^{n\delta}}
\]
for $M \gg 0$, we may assume that each of the $p^{n\delta}$
($R$-linear) component functions $R^{1/pn} \to R$ of $\mu$ are
$R^{1/p^{n}}$-multiples of $\phi$ and also that $\cokernel(\mu)$ is
annihilated by $h = c^{N}$ for some $N \gg 0$. This gives (a).

Let us now show how to iterate $\mu$ to get an inclusion $\mu_{m} \: R^{1/p^{mn}} \to R^{\oplus p^{nm\delta}}$ for all $m \in \Z_{>0}$ where $\cokernel(\mu_{m})$ is annihilated by $g = c^{2N}$ and each of the component functions of $\mu_{m}$ are $R^{1/p^{nm}}$-multiples of $\phi^{m}$. We begin by setting $\mu_{1} := \mu$ and recursively define $\mu_{m+1} := (\mu^{\oplus p^{n\delta}}) \circ ( \mu_{m}^{1/p^{n}} )$, so that we have
\[
\xymatrixcolsep{6pc}
\xymatrix{
R^{1/p^{(m+1)n}} \ar@/_1.5pc/[rr]_{\mu_{m+1}} \ar[r]^-{\mu_{m}^{1/p^{n}}}& (R^{1/p^{n}})^{\oplus
  p^{nm\delta}} \ar[r]^-{\mu^{\oplus p^{n\delta}}} &
R^{\oplus p^{n(m+1)\delta}}
}.
\]
Arguing inductively, we may assume $\cokernel(\mu_{m})$ is killed by $c^{2N}$ and the component functions of $\mu_{m}$ are all $R^{1/p^{mn}}$-multiples of $\phi^{m}$.  Thus, we have that $\cokernel(\mu_{m}^{1/p^{n}})$ is killed by $c^{2N/p^{n}}$ ($\mu_{m}^{1/p^{n}}$ is $R^{1/p^{n}}$-linear) and hence also by $c^N$.  As $\cokernel(\mu)$ is annihilated by $c^{N}$, it follows immediately that $\cokernel(\mu_{m+1})$ is annihilated by $c^{2N}$. Furthermore, the component functions of $\mu_{m+1}$ are $R^{1/p^{n(m+1)\delta}}$-multiples of $\phi^{m+1}$ as they are compositions of the component functions of $\mu_{m}^{1/p^{n}}$ (which are multiples of $(\phi^{m})^{1/p^{n}}$) with the component functions of $\mu$.  This gives (b).

Finally, suppose now that $\psi \in \Hom_{R}(R^{1/p^{nm}},R)$ for some $m > 0$.  Then $g \cdot \psi$ must factor through $\mu_{m}$ as $g$ annihilates $\cokernel(\mu_{m})$, and hence $g \cdot \psi$ can be written as an $R$-linear combination of the component functions of $\mu_{m}$.  In particular, we have that $g \cdot \psi$ is a \mbox{$R^{1/p^{nm}}$-multiple} of $\phi^{m}$ for all $\psi \in \Hom_{R}(R^{1/p^{nm}},R)$.
\end{proof}

\begin{lemma}
\label{lem:onemapbound}
  Suppose $\sD$ is a Cartier subalgebra on a $d$-dimensional $F$-finite local domain
 $(R,\m)$.  Set $\delta = d + \alpha(R)$ and $s_{e'}^{\sD} = \frac{a_{e'}^{\sD}}{p^{e'\delta}} = \frac{1}{p^{e'd}}\length_{R}(R/I_{e'}^{\sD})$ for all $e' \in \Z_{>0}$.
 If $0 \neq \phi \in \sD_{e}$ for some $e \in \Z_{>0}$, then there exists a
 positive constant
 $C_{\phi} \in \Z_{>0}$ (depending only on $\phi$) such that
\[
s_{e'}^{\sD} \leq s_{e'+be}^{\sD} + \frac{C_{\phi}}{p^{e'}}
\]
for all $e' \in \Z_{>0}$ and all $b \in \Z_{>0}$.
\end{lemma}

\begin{proof}
  By Lemma~\ref{lem:directsumstuff}~(a), there exists $0 \neq h \in R$
and an inclusion $\mu \: F^{e}_{*}R \to
R^{\oplus p^{e\delta}}$ of $R$-modules  where $h$ annihilates
$\cokernel(\mu)$ and the component functions of $\mu$ are
each \mbox{$F^{e}_{*}R$-multiples} of $\phi$.  In particular, each of
these component functions is in $\sD_{e}$, and it follows from
Lemma~\ref{lem:propsofIe} that $\mu(F^{e}_{*}I_{e'+e})
\subseteq I_{e'}^{\oplus p^{e\delta}}$ for all $e' \in \Z_{>0}$.
 Thus, $\mu$ induces a homomorphism of $R$-modules
$
F^{e}_{*}(R/I_{e'+e}^{\sD})  \to (R/I_{e'}^{\sD})^{\oplus
     p^{e\delta}}$
whose cokernel is once again annihilated by $h$.  In particular,
since $\m^{[p^{e'}]} \subseteq I_{e'}^{\sD}$, this cokernel is a quotient of $(R/\langle h, \m^{[p^{e'}]} \rangle)^{\oplus
     p^{e\delta}}$.  Let $D$ be a positive constant such that $\length_{R}(R/\langle h, \m^{[p^{e'}]}\rangle) \leq Dp^{e'(d-1)}$ for all $e' \in \Z_{>0}$.  It follows that
\[
\begin{array}{rcl}
p^{e\delta} \length_{R}(R/I_{e'}^{\sD}) &=& \length_{R}\left((R/I_{e'}^{\sD})^{\oplus p^{e\delta}}\right) \\ &\leq& \length_{R}\left( F^{e}_{*}(R/I_{e'+e}^{\sD})\right) + \length_{R}\left( (R/\langle h, \m^{[p^{e'}]} \rangle)^{\oplus
     p^{e\delta}}\right) \\ & \leq & p^{e\alpha(R)} \length_{R}(R/I_{e'+e}^{\sD}) + Dp^{e\delta +e'(d-1)}
\end{array}
\]
and dividing through by $p^{e\delta + e'd}$ gives
\[
s_{e'}^{\sD} = \frac{1}{p^{e'd}}\length_{R}(R/I_{e'}^{\sD}) \leq \frac{1}{p^{(e'+e)d}}\length_{R}(R/I_{e'+e}^{\sD}) + \frac{D}{p^{e'}} = s_{e'+e}^{\sD} + \frac{D}{p^{e'}} \, \, .
\]
Iterating this inequality now gives
\[
s_{e'}^{\sD} \leq s_{e' + e}^{\sD} + \frac{D}{p^{e'}}
\leq s_{e'+2e}^{\sD} + \frac{D}{p^{e'}}\left(1 + \frac{1}{p^{e}} \right) \leq \cdots \leq s_{e'+be}^{\sD} + \frac{D}{p^{e'}}\left(1 + \frac{1}{p^{e}} + \cdots + \frac{1}{p^{(b-1)e}} \right) \, \, .
\]
Thus, setting $C_{\phi} = 2D$, we have
\[
s_{e'}^{\sD} \leq s_{e' + be}^{\sD} + \frac{D}{p^{e'}}\left( \frac{1}{1-\frac{1}{p^{e}}}\right) = s_{e' + be}^{\sD} + \frac{D}{p^{e'}}\left( 1 + \frac{1}{p^{e}-1}\right) \leq s^{\sD}_{e' + be} + \frac{C_{\phi}}{p^{e'}}
\]
as desired.
\end{proof}
Before proceeding with the proof of the existence theorem we now quickly show that the $F$-signature is zero if $(R,\sD)$ is not $F$-regular.
\begin{lemma}
\label{lem:notfreg}
  Suppose $\sD$ is a Cartier subalgebra on a $d$-dimensional $F$-finite local ring
 $(R,\m)$.  If $(R, \sD)$ is not $F$-regular, then there exists a positive constant $C \in \Z_{>0}$ such that
\[
\frac{a_{e}^{\sD}}{p^{e(d+\alpha(R))}} = \frac{1}{p^{ed}}\length_{R}(R/I_{e}^{\sD}) \leq \frac{C}{p^{e}}
\]
for all $e \in \Z_{>0}$.  In particular, $\displaystyle s(R, \sD) = \lim_{e \in \Gamma_{\sD}} \frac{a_{e}^{\sD}}{p^{e(d+\alpha(R))}} = 0$.
\end{lemma}

\begin{proof}
As $(R, \sD)$ is not $F$-regular, there is some $c \in R$ not
  contained in any minimal prime with $c \in I_{e}^{\sD}$ for all $e
  \in \Z_{>0}$.  In particular, $\length_{R}(R/I_{e}^{\sD}) \leq
  \length_{R}(R/\langle \m^{[p^{e}]}, c \rangle) \leq C p^{e(d-1)}$
  for some positive constant $C$.  Dividing through by $p^{ed}$ gives the desired inequality, and it follows immediately that the limit $ s(R, \sD) = \lim_{e \in \Gamma_{\sD}} \frac{a_{e}^{\sD}}{p^{e(d+\alpha(R))}}$ exists and equals zero.
\end{proof}

\newcommand{\notk}{t}
\begin{proof}[Proof of Theorem~\ref{thm:existence}]
Set $\delta = d + \alpha(R)$ and $s_{e'}^{\sD} = \frac{a_{e'}^{\sD}}{p^{e'\delta}} = \frac{1}{p^{e'd}}\length_{R}(R/I_{e'}^{\sD})$ for all $e' \in \Z_{>0}$.  By Lemma~\ref{lem:notfreg} above, we are free to assume that $(R, \sD)$ is $F$-regular so that $R$ is in fact a domain.
Let $e_{1}, \ldots, e_{\notk} \in \Z$ be a set of generators for $\Gamma_{\sD}$ and $M \gg 0$ so that $mn_{\sD} \in \Gamma_{\sD}$ for all $m \geq M$.
For each $i = 1, \ldots, \notk$, choose a non-zero $p^{-e_{i}}$-linear
map $\phi_{i} \in \sD_{e_{i}}$. By
Lemma~\ref{lem:onemapbound}, there exists constants $C_{\phi_{i}} \in \Z_{>0}$
with $s_{e'}^{\sD} \leq s_{e' + be_{i}}^{\sD} +\frac{C_{\phi_{i}}}{p^{e'}}$ for all $e',b \in \Z_{>0}$.  Let $D = \max\{C_{\phi_{1}}, \ldots, C_{\phi_{\notk}}\}$ and set $C = 2D$. Consider now any $e \in \Gamma_{\sD}$, and write $e = \sum_{i = 1}^{\notk} b_{i}e_{i}$ for some $b_{i} \in \Z_{\geq 0}$. We have
\[
\begin{array}{rcl}
  s_{e'}^{\sD} &\leq& \displaystyle s_{e'+b_{1}e_{1}}^{\sD} +
  \frac{D}{p^{e'}} \\
&\leq & \displaystyle s_{e'+b_{1}e_{1} + b_{2}e_{2}}^{\sD} +
\frac{D}{p^{e'}}\left( 1 + \frac{1}{p^{b_{1}e_{1}}} \right) \medskip \\
&\vdots & \quad \quad \qquad \vdots \quad \qquad \qquad \vdots
\medskip \\
& \leq & \displaystyle s_{e' +  b_{1}e_{1}+ \cdots + b_{\notk}e_{\notk}}^{\sD} + \frac{D}{p^{e'}}\left(
  1 + \frac{1}{p^{b_{i}e_{i}}} + \cdots + \frac{1}{p^{b_{1}e_{1} +
      \cdots + b_{\notk-1}e_{\notk-1}}} \right) \quad \leq \quad s_{e +
  e'}^{\sD} + \frac{C}{p^{e'}} \, \, .
\end{array}
\]
In particular, if $m_{1} \in \Z_{> 0}$ and $m_{2} \in \Z_{\geq M}$, we have
\[
s_{m_{1}n_{\sD}}^{\sD} \leq s^{\sD}_{(m_{1}+m_{2})n_{\sD}} +
\frac{C}{p^{m_{1}n_{\sD}}} \, \, .
\]
Considering $m_{1}$ fixed and letting $m_{2} \to \infty$, we may conclude
\[
s_{m_{1}n_{1}}^{\sD} \leq \liminf_{m \to \infty} s^{\sD}_{mn_{\sD}} + \frac{C}{p^{m_{1}n_{\sD}}}
\]
for all $m_{1} \in \Z_{> 0}$.  Now, letting $m_{1} \to \infty$ gives
\[
\limsup_{m \to \infty} s_{mn_{\sD}}^{\sD} \leq \liminf_{m \to \infty}
s^{\sD}_{mn_{\sD}}
\]
from which it follows that the desired limit
\[
\lim_{m \to \infty}
s_{mn_{\sD}}^{\sD} = \lim_{m \to \infty}
\frac{a_{mn_{\sD}}^{\sD}}{p^{mn_{\sD}(d+\alpha(R))}} = \lim_{e \in \Gamma_{\sD} \to \infty}
\frac{a_{e}^{\sD}}{p^{e(d+\alpha(R))}} = s(R, \sD)
\]
exists.
\end{proof}

\begin{remark}
  While Theorem~\ref{thm:existence} recovers the existence results of \cite{TuckerFSignatureExists} (using similar methods), the above proof is notable in that it does not make use of Hilbert-Kunz multiplicity.  Nonetheless, \textit{a posteriori}, the following result follows immediately from \cite[Corollary 3.7]{TuckerFSignatureExists} after establishing the existence of the $F$-signature limit.
\end{remark}

\begin{corollary}
  Suppose $\sD$ is a Cartier subalgebra on a $d$-dimensional $F$-finite local ring
 $(R,\m)$.  Then
\[
s(R, \sD) = \lim_{e \in \Gamma_{\sD}\to \infty} \frac{1}{p^{ed}} e_{HK}( I_{e}^{\sD})
\]
\end{corollary}

\subsection{Positivity}
\label{sec:gener-f-sign-pos}
The main result of this section, stated immediately below, is an important justification of our generalized definition of $F$-signature.

\begin{theorem}
  \label{thm:positivity}
Suppose $\sD$ is a Cartier subalgebra on a $d$-dimensional $F$-finite local ring
 $(R,\m)$.  Then the $F$-signature of $(R,
\sD)$ is positive if and only if  $(R, \sD)$ is $F$-regular.
\end{theorem}

\begin{remark}
  The above Theorem recovers and vastly generalizes the result of I. Aberbach and G. Leuschke
in \cite{AberbachLeuschke} showing that the positivity of the
$F$-signature of a ring is equivalent to $F$-regularity.  Even in this
classical setting, however, our proof is inherently quite different.
Nevertheless, many of the
ideas to follow are rooted in the arguments they present.
\end{remark}

If $R$ is a domain, recall that $R^{+}$ denotes
the absolute integral closure of $R$, \textit{i.e.} the integral
closure of $R$ inside an algebraic closure of its fraction field.
Before proceeding, we first state a needed result of
M. Hochster and C. Huneke,  \cite[Theorem 3.3]{HochsterHunekeTightClosureElementsSmallOrder}, whose proof relies heavily on the use of tight
closure techniques.
Our formulation varies from their precise statement
only in that we explicitly identify (by tracing through the
original argument using a fixed system of parameters
  $x_{1}, \ldots, x_{d}$) the
  constant $\beta$ which is merely stated to exist in their presentation.

\begin{theorem}
\label{thm:hochhunvalsthm}
  \cite[Theorem 3.3]{HochsterHunekeTightClosureElementsSmallOrder}
Let $(R,\m)$ be a
  complete local domain and $\nu$ a $\Q$-valued valuation on $R^{+}$
  non-negative on $R$ (and, hence, on $R^{+}$) and positive on $\m$
  (and, hence, on $\m^{+}$). Suppose that $x_{1}, \ldots, x_{d}$ is a system of
  parameters and set $\beta := \min_{i} \nu(x_{i}) $.  Then there exists a
  positive integer $r$ such that for every element $u$ of $R^{+}$ with
  $\nu(u) < \beta$ there is an $R$-linear map $\Theta \: R^{+} \to R$
  where $\Theta(u) \not\in \m^{r}$.
\end{theorem}

The main ingredient in the proof of Theorem~\ref{thm:positivity} is
the following result, which should be compared with \cite[Proposition 2.4]{AberbachExtensionofFRegularRingsbyFlatMaps}.

\begin{theorem}
  \label{thm:keylemmaforpositivity}
Suppose $(R,\m)$ is a $d$-dimensional $F$-finite complete local domain
 and $\phi \in \sC_{n} =
 \Hom_{R}(R^{1/p^{n}},R)$ is a non-zero $p^{-n}$-linear map for some $n
 \in \Z_{>0}$.  Then there exist $r, e_{0} \in \Z_{>0}$
with the following
property: for each $m \in \Z_{>0}$, every ideal $J$
with $\phi^{m}(J^{1/p^{mn}}) \subseteq \m^{r}$ satisfies $J \subseteq
\m^{\lfloor p^{mn - e_{0}} \rfloor}$.
\end{theorem}

\begin{proof}
Let $0 \neq g \in R$ satisfy the conclusion of Lemma~\ref{lem:directsumstuff}~(c) for $\phi$, so that for all $m \in \Z_{>0}$ and $\psi \in \sC_{nm} = \Hom_{R}(R^{1/p^{nm}},R)$ we have that $g \cdot \psi$
is a $R^{1/p^{nm}}$-multiple of $\phi^{m}$.
Following the argument of \cite[Proposition
2.4]{AberbachExtensionofFRegularRingsbyFlatMaps}, set $f(x)$ to be the largest power of $\m$ containing a given $x \in
  R$ and put $\bold{f}(x) = \lim_{n\to \infty} \frac{1}{n}f(x^{n})$.
  By the valuation theorem of Rees \cite[Theorem
  4.16]{ReesAsymptoticTheoryofIdeals}, there exists a finite number of
  $\Z$-valued valuations $\nu_{1}, \ldots, \nu_{m}$ on $R$ which are
non-negative on $R$ and positive on $\m$ together with positive
rational numbers $N_{1}, \ldots, N_{m}$ such that $\bold{f}(x) =
\min_{i} \nu_{i}(x)/N_{i} $.  Further applying \cite[Theorem
  5.32]{ReesAsymptoticTheoryofIdeals}, there is a constant $L \in
  \Z_{\geq 0}$ such that
$
f(x) \leq \lfloor \bold{f}(x) \rfloor \leq f(x) + L
$
for all $x \in R$.  In particular, $f(x) \geq \bold{f}(x) - L - 1$.

We can find a system of parameters $x_{1}, \ldots, x_{d}$ of $R$ such that
$
\beta_{i} := \min\nolimits_{j} \nu_{i}(x_{j}) > \nu_{i}(g)
$
for all $i$.  Indeed, since each $\nu_{i}$ is positive on
$\m$,  this can be achieved by taking sufficiently
large powers of any given system of parameters.  Now, each $\nu_{i}$
extends (non-uniquely) to a
\mbox{$\Q$-valued} valuation (also denoted $\nu_{i}$) on $R^{+}$ which is non-negative on $R^{+}$ and
positive on $\m^{+}$.
For each $\nu_{i}$ and $\beta_{i}$, we can find $r_{i} \in
\N$ satisfying the  conclusion of Theorem~\ref{thm:hochhunvalsthm}.  Set $r :=
\max_{i} r_{i}$.

\begin{claim*}
  If $h \in R$ satisfies $\phi^{m}((hR)^{1/p^{nm}}) \subseteq
  \m^{r}$ for some $m \in \Z_{>0}$, then for all $i$ we must have
$
\nu_{i}(h) \geq (\beta_{i} - \nu_{i}(g)) p^{nm}$.
\end{claim*}

\noindent
\textit{Proof of Claim.}
  Suppose, by way of contradiction, there is some index $i$ for which
  we have $
    \nu_{i}(h) < (\beta_{i} - \nu_{i}(g))p^{nm}$.  Then it
    follows that $\nu_{i}((g^{p^{nm}})h) < \beta_{i} p^{nm}$, and taking
 $p^{nm}$-th roots yields
$\nu_{i}(g (h^{1/p^{nm}}) ) < \beta_{i}$.  From
Theorem~\ref{thm:hochhunvalsthm} above, there is an $R$-linear map $\Theta \: R^{+} \to R$ with
$\Theta(g(h^{1/p^{nm}})) \not\in \m^{r}$ (since
$\m^{r} \subseteq \m^{r_{i}}$).  Now, $\Theta|_{R^{1/p^{nm}}} \in
\Hom_{R}(R^{1/p^{nm}},R)$, whence $g \cdot \Theta|_{R^{1/p^{nm}}} =
\phi^{m} \cdot r^{1/p^{nm}}$ for some $r \in R$ (by the
construction of $g$ above).  However, we then
arrive at the contradiction $
\Theta(g(h^{1/p^{nm}})) = \phi^{m}((hr)^{1/p^{nm}}) \in
\phi^{m}((hR)^{1/p^{nm}}) \subseteq \m^{r}$.

\medskip

To conclude the proof of Theorem~\ref{thm:keylemmaforpositivity}, suppose now that $J$ is an ideal with $\phi^{m}(J^{1/p^{nm}})
\subseteq \m^{r}$ for some $m \in \Z_{>0}$.  Let $\alpha = \min_{i}
\frac{\beta_{i}-\nu_{i}(g)}{N_{i}}$.
For any $h \in J$, it then follows from the above Claim that
\[
f(h) \geq  \bold{f}(h)  - L - 1   =
 \min_{i} \frac{\nu_{i}(h)}{N_{i}}  - L -1
  \geq  \alpha p^{nm}  - L -1 \; .
\]
Choosing $e_{0} \in \Z_{>0}$ sufficiently large, we have
$\lfloor p^{nm - e_{0}} \rfloor \leq \max \{ 0,  \alpha p^{nm}
- L  -1 \}$ for all $m \in \Z_{>0}$.  It now follows that $J
\subseteq \m^{\lfloor p^{nm - e_{0}} \rfloor}$ for all $m \in \Z_{>0}$.
\end{proof}

\begin{proof}[Proof of Theorem~\ref{thm:positivity}]
  By Lemma~\ref{lem:notfreg}, we need only show $s(R, \sD)>0$ assuming $(R,\sD)$ is $F$-regular.  Furthermore, using Corollary~\ref{cor:completionhassamenumbers}, the statement reduces immediately to the case
  where $R$ is a complete local domain.

Since $(R, \sD)$ is also $F$-pure, fix a surjective $p^{-n}$-linear map $\phi \in \sD_{n}$ for some $n \in \Z_{>0}$. For any $e \in \Gamma_{\sD}$, we have $mn - e \in \Gamma_{\sD}$ for some $m \gg 0$ (so that $\sD_{mn-e}$ contains a surjective map).  It follows that $I_{mn}^{\sD} \subseteq I^{\sD}_{mn - (mn - e)} = I_{e}^{\sD}$ by Lemma \ref{lem:propsofIe}.  Now, as $(R, \sD)$ is $F$-regular, we may conclude
\[
0 = P_{\sD} = \bigcap_{e>0} I_{e}^{\sD} = \bigcap_{e \in \Gamma_{\sD}} I_{e}^{\sD}= \bigcap_{m > 0} I_{nm}^{\sD} \, \, .
\]
In addition, using that $\phi$ is surjective, we see the sequence of ideals $\{I_{nm}^{\sD}\}_{m \in \Z_{>0}}$ is non-increasing.  Thus, by Chevalley's Lemma, this sequence is cofinal with the powers of $\m$.

Let $r,e_{0} \in \Z_{>0}$ satisfy the conclusion of
 Theorem~\ref{thm:keylemmaforpositivity} for $\phi$, and fix $m_{0} \in \Z_{>0}$ with
$I_{nm_{0}}^{\sD} \subseteq \m^{r}$.  This implies
that
$
\phi^{m}((I_{n(m+m_{0})}^{\sD})^{1/p^{nm}}) \subseteq
I_{nm_{0}}^{\sD} \subseteq \m^{r}
$
for all $m \in \Z_{>0}$ by Lemma \ref{lem:propsofIe}, so by Theorem~\ref{thm:keylemmaforpositivity}
we conclude $I_{n(m+m_{0})}^{\sD} \subseteq \m^{\lfloor p^{mn -
    e_{0}}\rfloor}$ for all $m \in \Z_{>0}$.  Thus, we have
\[
\begin{array}{rcccl}\displaystyle
  s(R, \sD) & = &\displaystyle \lim_{m \to \infty}
  \frac{a_{n(m+m_{0})}^{\sD}}{p^{n(m+m_{0})(d+\alpha(R))}} & = &
  \displaystyle \lim_{m \to \infty} \frac{\length_{R}(R /
    I_{n(m+m_{0})}^{\sD})}{p^{n(m+m_{0})d}} \\
& \geq & \displaystyle \lim_{m \to \infty} \frac{\length(R/\m^{\lfloor
  p^{mn - e_{0}}\rfloor})}{p^{n(m+m_{0})d}} & =
&\displaystyle \frac{1}{p^{(nm_{0}+e_{0})d}} \lim_{m\to \infty} \frac{\length(R/\m^{\lfloor
  p^{mn - e_{0}}\rfloor})}{p^{(nm - e_{0})d}} \\
& = & \displaystyle \frac{1}{p^{(nm_{0}+e_{0})d}} \cdot \frac{e(R)}{d!} & > & 0
\end{array}
\]
where $e(R) \geq 1$ denotes the Hilbert-Samuel multiplicity of $R$.
\end{proof}

\section{Applications}
\label{sec:applications}

\subsection{$F$-splitting ratio}
\label{sec:gener-f-splitt}

Suppose $\sD$ is a Cartier subalgebra on a $d$-dimensional $F$-finite local ring
 $(R,\m,k)$.  When
$(R, \sD)$ is $F$-regular, the $F$-signature precisely characterizes
the growth rate of the $F$-splitting numbers $a_{e}^{\sD}$.
However, when $(R, \sD)$ is not $F$-regular,
$s(R, \sD) = \lim_{e \in \Gamma_{\sD}}
\frac{a_{e}^{\sD}}{p^{e(d + \alpha(R))}} = 0$ merely gives a (crude) upper bound.
The aim of this
section is to give a precise characterization of the
asymptotic behavior of
$a_{e}^{\sD}$ for $e \gg 0$.  The main idea is rather simple: Cartier subalgebras make it possible to pass to the appropriate quotient ring.
We begin with a Lemma facilitating this transition.

\begin{lemma} [\cf \cite{SchwedeFAdjunction}]
\label{lem:passtoquotients}
   Suppose $\sD$ is a Cartier subalgebra on an $F$-finite local ring
 $(R,\m)$.  If $J \subseteq R$ is a $\sD$-compatible ideal and $\sD_{J}$
is the induced Cartier subalgebra on $R/J$ as in
Definition~\ref{def:compatideal}, then $I_{e}^{\sD_{J}} = I_{e}^{\sD} /J$ and
$a_{e}^{\sD_{J}} = a_{e}^{\sD}$ for all $e \in \Z_{>0}$.  In particular, we have
 $\Gamma_{\sD_{J}} =
\Gamma_{\sD}$ and $n_{\sD_{J}} = n_{\sD}$.

\end{lemma}

 \begin{proof}
\renewcommand{\iff}{\; \Leftrightarrow \;}
Fix $e \in \Z_{>0}$ and consider $r + J \in (R/J)$.  Since we have $\phi(F^{e}_{*}r) + J  = \phi_{J}(F^{e}_{*}(r + J))$ for all $\phi \in \sD_{e}$ and also $(\sD_{J})_{e} = \{ \, \phi_{J} \; | \; \phi \in \sD_{e} \, \}$, it follows that
\[
\begin{array}{ccccc}
r \in I_{e}^{\sD} & \iff &  \phi(F^{e}_{*}r) \in \m \mbox{ for all } \phi \in \sD_{e}  & & \\
& \iff &  \phi_{J}\left(F^{e}_{*}(r + J) \right) \in \bm/J \mbox{ for all } \phi_{J} \in (\sD_{J})_{e}  &\iff & (r + J) \in I_{e}^{\sD_{J}}
\end{array}
\]
and in particular $I_{e}^{\sD_{J}} = I_{e}^{\sD} /J$ (note that $I_{e}^{\sD}$ contains every proper $\sD$-compatible ideal).  Thus, we may conclude that $a_{e}^{\sD} = {p^{e\alpha(R)}}\length_{R}(R/I_{e}^{\sD}) = {p^{e \alpha(R/J)}} \length_{R/J}(R/J \big/I_{e}^{\sD}/J) = a_{e}^{\sD_{J}}$ and so also both $\Gamma_{\sD_{J}} = \Gamma_{\sD}$ and $n_{\sD_{J}} = n_{\sD}$ as well.
 \end{proof}

 \begin{theorem}
   \label{thm:splittingratio}
Consider a Cartier subalgebra $\sD$ on an $F$-finite local ring
 $(R,\m)$.  Suppose $(R, \sD)$ is $F$-pure with $F$-splitting prime $P_{\sD}$.  Then the limit
\[
r_{F}(R, \sD) = \lim_{e \in \Gamma_{\sD} \to \infty}
\frac{a_{e}^{\sD}}{p^{e\left(\dim(R/P_{\sD}) + \alpha(R)\right)}} =
\lim_{m \to \infty}
\frac{a_{mn_{\sD}}^{\sD}}{p^{mn_{\sD}(\dim(R/P_{\sD}) +\alpha(R))}}
\]
exists and is called the $F$-splitting ratio of $(R, \sD)$.  Furthermore, if $\sD_{P_{\sD}}$ denotes the Cartier subalgebra on $R/P_{\sD}$ induced by $\sD$ as in Definition~\ref{def:compatideal}, we have
\[
r_{F}(R, \sD) = s( R/P_{\sD}, \sD_{P_{\sD}})
\]
so that, in particular, the $F$-splitting ratio
satisfies
$0 < r_{F}(R, \sD) \leq 1$.
 \end{theorem}

\begin{proof}
From Lemma~\ref{lem.BijectionBetweenCompatibleIdeals}, we have that $(\overline{R}, \overline{\sD})$ is $F$-regular and by Lemma~\ref{lem:passtoquotients} we have
\[
r_{F}(R, \sD) = \lim_{e \in \Gamma_{\sD} \to \infty}
\frac{a_{e}^{\sD}}{p^{e(\dim(\overline{R} + \alpha(R)))}} = \lim_{e \in \Gamma_{\overline{\sD}} \to \infty}
\frac{a_{e}^{\overline{\sD}}}{p^{e(\dim(\overline{R} + \alpha(R)))}} =
s(\overline{R}, \overline{\sD}) \, \, .
\]
In particular, Theorem~\ref{thm:existence} gives that
the limit under consideration exists and Theorem~\ref{thm:positivity}
further implies $0 < s(\overline{R}, \overline{\sD}) \leq 1$.
 \end{proof}

 \begin{corollary}
   \label{cor:fsplitratio}
   Suppose $R$ is an $F$-finite local ring. If $R$ is $F$-pure with $F$-splitting prime $P_{R}$, then
\[
r_{F}(R) = \lim_{e \to \infty} \frac{a_{e}^{R}}{p^{e(\alpha(R) + \dim(R/P_{R}))}} > 0.
\]
 \end{corollary}

 \begin{corollary}
   Suppose $(R,\m)$ is an $F$-finite local ring
 and
$\sD$ is an $F$-pure Cartier subalgebra on $R$ with $F$-splitting prime $P_{\sD}$.
Then the limit
\[
\lim_{e\in \Gamma_{\sD} \to \infty}
  \frac{\log_{p}a_{e}^{\sD}}{e} = \dim(R / P_{\sD}) + \alpha(R)  \in
\Z_{\geq 0}
\]
exists and is a non-negative integer.
 \end{corollary}

\begin{definition}
  Suppose $(R,\m)$ is an $F$-finite local ring
 and
$\sD$ is an $F$-pure Cartier subalgebra on $R$ with $F$-splitting prime $P_{\sD}$.
The \emph{$F$-splitting dimension of $(R, \sD)$}, denoted $\sdim(R, \sD)$, is the common value of the following equivalent expressions:

\begin{center}
\begin{tabular}{l@{\hspace{3mm}}c@{\hspace{.5cm}}l@{\hspace{3mm}}c}
(a) &$\sup \left\{ \; \delta \in \R  \; \; \left| \displaystyle \; \; \liminf_{e \in \Gamma_{\sD} \to \infty}
 \frac{a_{e}^{\sD}}{p^{e(\delta+\alpha(R))}} > 0 \; \right. \right\}$ & (c) &$\dim(R / P_{\sD})$
\\
(b) &
$\max \left\{ \; \delta \in \Z  \; \; \left| \displaystyle \; \; \lim_{e \in \Gamma_{\sD} \to \infty}
 \frac{a_{e}^{\sD}}{p^{e(\delta+\alpha(R))}} > 0 \; \right. \right\}$
& (d) &  $\displaystyle \lim_{e \in \Gamma_{\sD} \to \infty} \left( \frac{\log_{p}
    a_{e}^{\sD}}{e} - \alpha(R) \right)$ \smallskip \\
\end{tabular}
\end{center}
Note that the equivalence of (a) - (d) follows immediately from Theorem~\ref{thm:splittingratio}, and in particular the expression in (a) takes only integer values.
\end{definition}

\begin{remark}
\label{rmk:aeremark}
The $F$-splitting dimension of a local ring $R$ is simply the
  $F$-splitting dimension of $(R, \sC^{R})$ and was
  first introduced by I. Aberbach and F. Enescu
  \cite{AberbachEnescuStructureOfFPure} using the characterization in (b).
  The equality of (b) and (c) in this case gives a positive answer to \cite[Question 4.9]{AberbachEnescuStructureOfFPure}.
\end{remark}

\subsection{$F$-graded systems of ideals}
\label{sec:f-graded-systems}

In this section, we develop additional tools for the computation of \mbox{$F$-signature}.  In particular, we develop methods for computing $F$-splitting numbers by studying certain colon ideals on regular rings, similar to \cite[Lemma 1.6]{FedderFPureRat}.  This Fedder-type criterion makes use of $F$-graded systems of ideals as defined in \cite{BlickleTestIdealsViaAlgebras}.

\renewcommand{\n}{\mathfrak{n}}
\renewcommand{\b}{\mathfrak{b}}
\renewcommand{\a}{\mathfrak{a}}
\begin{definition}
Suppose $(S, \n$ is an $F$-finite local ring.  A \emph{$F$-graded system of ideals of $S$} is a sequence of ideals $\b_{\bullet} = \{ \b_{e} \}_{e \in \Z_{\geq 0}}$ such that $\b_{e}^{[p^{l}]} \b_{l} \subseteq \b_{e+l}$ for all $e,l \in \Z_{\geq 0}$.  To avoid pathologies we will assume that $\b_{0} = R$ and $\b_{e} \neq 0$ for some $e > 0$.
\end{definition}

\begin{remark}
Recall that a (standard) graded system of ideals (\textit{cf.} \cite[Definition 1.1]{EinLazSmithSymbolic}) is a sequence of ideals $\a_{\bullet} = \{ \a_{n} \}_{n \in \Z_{\geq 0}}$ such that $\a_{n} \a_{n'} \subseteq \a_{n+n'}$ for all $n, n' \in \Z_{\geq 0}$.  This naturally gives rise to a $F$-graded system by taking $\b_{e} = \a_{p^{e}-1}$.  Roughly speaking, this $F$-graded system keeps track of $\a_{n}$ for large values of $n$ which are not divisible by $p$.
\end{remark}

\begin{lemma}\label{lem:idealstoalgebras}
If $(S, \n )$ is an $F$-finite local ring, then every $F$-graded system of ideals $\b_{\bullet}$ of $S$ defines a Cartier subalgebra $\sC^{\b_{\bullet}} = \oplus_{e \geq 0} (\sC^{\b_{\bullet}})_{e}$ on $S$ by setting $(\sC^{\b_{\bullet}})_{e} \colonequals \sC^{S}_{e} \cdot F^{e}_{*}\b_{e}$ for all $e \in \Z_{\geq 0}$.  If in addition $S$ is Gorenstein, then every Cartier subalgebra $\sD$ arises uniquely in this manner.
\end{lemma}

\begin{proof}
It is immediate to check that $\sC^{\b_{\bullet}}$ is closed under multiplication in $\sC^{S}$, giving the first statement (\textit{cf.} \cite[Lemma 4.1]{SchwedeBetterFPureFRegular}).  For the latter statement when $S$ is Gorenstein, one simply makes use of the fact that $\sC^S_e = \Hom_S(F^e_* S, S) \cong F^e_* S$ as $F^e_* S$-modules as in Example~\ref{ex:gorensteincartieralgebra}.
\end{proof}

\begin{definition}
  If $(S, \n)$ is an $F$-finite local ring and $\b_{\bullet}$ is an $F$-graded system of ideals of $S$, the $F$-splitting numbers of
  $(S, \b_{\bullet})$ are simply the $F$-splitting numbers of the associated Cartier subalgebra $\sC^{\b_{\bullet}}$ on $S$ as in Lemma~\ref{lem:idealstoalgebras}.  One defines the $F$-signature, $F$-splitting prime, and $F$-splitting ratio of $(S, \b_{\bullet})$ in a similar manner.  We shall conflate the pairs $(S, \b_{\bullet})$ and $(S, \sC^{\b_{\bullet}})$ in the notation and terminology appearing throughout.
\end{definition}

\begin{example}
\label{ex.TwoExamplesOfFGraded}
  The two main examples of $F$-graded systems we will make use of are
  the following.
  \begin{enumerate}
  \item If $J$ is any non-zero ideal of $S$, set $\b_{e} =
    (J^{[p^{e}]}:J)$ for all $e > 0$.  In case $S$ is in fact regular,
    as we shall see in Corollary~\ref{cor:fedcritforsplitnumbersandsignature}, studying this $F$-graded system is essentially
    equivalent to studying the $F$-splitting properties of $R = S/J$.
  \item If $\a$ is any non-zero ideal of $S$ and $t \in \R_{\geq 0}$,
    we may consider $\b_{e} = \a^{\lceil t(p^{e}-1) \rceil}$ for all
    $e > 0$.  This $F$-graded system is closely related to the notion
    of $\a^{t}$-tight closure for ideals in $S$ as in
    \cite{HaraYoshidaGeneralizationOfTightClosure}; see also \cite{SchwedeSharpTestElements} and Section \ref{sec:divisor-pairs} below.
  \end{enumerate}
\end{example}

\begin{lemma} [\cf \cite{FedderFPureRat}]
\label{lem:compsinregrings}
  Suppose $(S, \n, l)$ is an $F$-finite regular local ring and $\b_{\bullet}$ is an $F$-graded system of ideals in $S$ with associated Cartier subalgebra $\sD = \sC^{\b_{\bullet}}$ as in Proposition~\ref{lem:idealstoalgebras}.  Then we have
\[
a_{e}^{\b_{\bullet}} = a_{e}^{\sD} = p^{e\alpha(S)} \length_{S}\left(S/(\n^{[p^{e}]}:\b_{e})\right) = p^{e\alpha(S)} \length\left( \b_{e} /(\b_{e} \cap \n^{[p^{e}]}) \right)
\]
for all $e > 0$, which also equals the maximal number of copies of $S$ in a direct sum decomposition of $S^{1/p^{e}}$ as an $S$-module that are contained in $\b^{1/p^{e}}$.  Furthermore, we have
\[
P_{\b_{\bullet}} := P_{\sD} = \bigcap_{e>0} \left( \n^{[p^{e}]} : \b_{e}\right)
\]
and an ideal $J \subseteq S$ is $\sD$-compatible if and only if $\b_{e} \subseteq (J^{[p^{e}]}:J)$ for all $e > 0$.
\end{lemma}

\renewcommand{\iff}{\; \; \Leftrightarrow \; \;}
\begin{proof}
Let $\Phi$ be an $S^{1/p}$-module generator of $\Hom_{S}(S^{1/p},S) = \sC^{S}_{1}$.  Since $S^{1/p^{e}}$ is a free \mbox{$S$-module} and  $\Phi^{e}$ is an $S^{1/p^{e}}$-module generator of $\Hom_{S}(S^{1/p^{e}},S) = \sC^{S}_{e}$, it follows that for all ideals $I,J \subseteq S$ one has
\[
\Phi^{e}(I^{1/p^{e}}) \subseteq J \iff  \phi(I^{1/p^{e}}) \subseteq J \mbox{ for all } \phi \in \sC^{S}_{e}  \iff I^{1/p^{e}} \subseteq J  S^{1/p^{e}} \iff I \subseteq J^{[p^{e}]}
\]
for each $e > 0$, see \cite[Lemma 1.6]{FedderFPureRat}.  In particular,
\[
x \in I_{e}^{\sD} \iff \Phi^{e}\left( (x\b_{e})^{1/p^{e}} \right) \subseteq \n \iff x \b_{e} \subseteq \n^{[p^{e}]} \iff x \in (\n^{[p^{e}]} : \b_{e})
\]
so that $I_{e}^{\sD} = (\n^{[p^{e}]} : \b_{e})$ for all $e > 0$.  Thus, we have $P_{\sD} = \cap_{e> 0} I_{e}^{\sD} = \cap_{e>0} (\n^{[p^{e}]} : \b_{e})$ and $a_{e}^{\sD} = p^{e\alpha(S)} \length_{s}(S/I_{e}^{\sD}) = p^{e\alpha(S)} \length_{S}\left(S/(\n^{[p^{e}]}:\b_{e})\right)$.  Now, by construction, the $S^{1/p^{e}}$-linear isomorphism $\sC^{S}_{e} \to[\simeq] S^{1/p^{e}}$ given by sending $\Phi^{e} \mapsto 1$  identifies $\sD_{e}$ with $\b_{e}^{1/p^{e}}$.  Thus, by Proposition~\ref{prop:charofsplittingnumbers}~(b), we also have that $a_{e}^{\sD}$ equals the maximal number of copies of $S$ in a direct sum decomposition of $S^{1/p^{e}}$ as an $S$-module that are contained in $\b^{1/p^{e}}$.  Moreover,
\[
\Phi^{e}(x^{1/p^{e}} \cdot \blank) \in \sD^{\mathrm{ns}}_{e} \iff x \in \b_{e} \mbox{ and } \Phi^{e}(\langle x \rangle^{1/p^{e}}) \subseteq \n \iff x \in \b_{e} \cap \n^{[p^{e}]}
\]
so that the above isomorphism identifies $\sD_{e}^{\mathrm{ns}}$ with $(\b_{e}\cap \n^{[p^{e}]})^{1/p^{e}}$.  It follows that $a_{e}^{\sD} = \length_{S}(\sD_{e}/\sD_{e}^{\mathrm{ns}}) =  \length_{S}\left( \b_{e}^{1/p^{e}} / (\b_{e} \cap \n^{[p^{e}]})^{1/p^{e}} \right) = p^{e\alpha(S)} \length\left( \b_{e} /(\b_{e} \cap \n^{[p^{e}]}) \right)$.  Lastly, we may conclude that $J$ is $\sD$-compatible if and only if we have
\[
 \Phi^{e}((\b_{e}J)^{1/p^{e}}) \subseteq J\iff \b_{e}J \subseteq J^{[p^{e}]} \iff \b_{e} \subseteq (J^{[p^{e}]}:J)
\]
for all $e > 0$.
\end{proof}

\begin{theorem}
\label{thm:FedderComputationForFGradedSystems}
  Suppose $(S, \n)$ is an $F$-finite regular local ring and
$R = S/J$ is a quotient ring for some non-zero ideal $J$ in $S$.  Then
every Cartier subalgebra $\sD$ on $R$ gives rise to an $F$-graded system $\b_{\bullet}$  of
ideals in $S$ such that $\b_{e} \subseteq (J^{[p^{e}]}:J)$ and
\[
\frac{1}{p^{e\alpha(R)}} a_{e}(R,\sD) =
\frac{1}{p^{e\alpha(S)}}a_{e}^{\b_{\bullet}} = \length_{S}\left( S \big/
  (\n^{[p^{e}]}:\b_{e})\right) = \length_{S}\left( \b_{e} \big/ (\b_{e} \cap
  \n^{[p^{e}]})\right)
\]
 for all $e
\in \Z_{>0}$.  In addition, we have
\[
P_{\sD} = P_{\b_{\bullet}} \cdot R = \left( \bigcap_{e>0}
(\n^{[p^{e}]}:\b_{e}) \right) \cdot R
\]
so that $\sdim(R, \sD) = \sdim(S,\b_{\bullet})$ and $r_{F}(R,\sD)
= r_{F}(S,\b_{\bullet})$.
\end{theorem}

\begin{proof}
For each $e>0$, if we have $\phi \in \sC_{e}^{S}$ with $\phi(J^{1/p^{e}}) \subseteq J$, we will denote the induced $p^{-e}$-linear map on $R = S/J$ as in Definition~\ref{def:compatideal} by $\phi_{J}$.  Letting
\[
\widetilde\sD_{e} = \{ \, \phi \in \sC_{e}^{S} \; | \; \phi(J^{1/p^{e}}) \subseteq J \mbox{ and } \phi_{J} \in \sD_{e} \, \} \, \, ,
\]
it is easy to see that $\widetilde\sD = \oplus_{e \geq 0} \widetilde\sD_{e}$ is a Cartier subalgebra on $S$ by using that $\sD$ is a Cartier subalgebra on $R$.  Using Lemma~\ref{lem:idealstoalgebras}, $\widetilde\sD$ comes from an $F$-graded system of ideals in $\b_{\bullet}$ on $S$.  Note that $J$ is $\widetilde\sD$-compatible, so that $\b_{e} \subseteq (J^{[p^{e}]}:J)$ for all $e > 0$ by Lemma~\ref{lem:compsinregrings}.  By Lemma~\ref{lem:passtoquotients}, we have $a_{e}^{\sD} = a_{e}^{\widetilde\sD} = a_{e}^{\b_{\bullet}}$ and $P_{\sD} = P_{\widetilde\sD} \cdot R = P_{\b_{\bullet}} \cdot R$,  so the remainder of the Theorem follows immediately from the calculations shown in Lemma~\ref{lem:compsinregrings}.
\end{proof}

The following corollary originally appeared in \cite[Discussion after Theorem 4.2]{AberbachEnescuStructureOfFPure} and in \cite[Proposition 3.1]{EnescuYaoLowerSemicontinuity} (where it is done in the non-$F$-finite case).  However, since it follows easily from what is done above we state it here to illustrate the meaning of Theorem~\ref{thm:FedderComputationForFGradedSystems}.

\begin{corollary}[Fedder-type criterion for $F$-splitting numbers and $F$-signature] \cite[Discussion after Theorem 4.2]{AberbachEnescuStructureOfFPure} \cite[Proposition 3.1]{EnescuYaoLowerSemicontinuity}
\label{cor:fedcritforsplitnumbersandsignature}
  If $(S, \n)$ is an $F$-finite regular local ring and $R = S/J$ is
  a quotient ring of dimension $d$ for some non-zero ideal $J$, then
  the $F$-splitting numbers of $R$ can be computed on $S$ as
\[
\frac{1}{p^{e\alpha(R)}}a_{e}(R) = \length_{S}\left( S \big/ \left(
    \n^{[p^{e}]}:(J^{[p^{e}]}:J)\right)\right) = \length_{S}\left(
  (J^{[p^{e}]}:J) \big/ \left( (J^{[p^{e}]}:J) \cap
    \n^{[p^{e}]}\right) \right) \, \, .
\]
In particular,
$F$-signature of $R$ can be expressed as
\[
s(R) = \lim_{e \to \infty} \frac{1}{p^{ed}}\length_{S}\left(  S \big/ \left(
    \n^{[p^{e}]}:(J^{[p^{e}]}:J)\right) \right) = \lim_{e \to \infty} \frac{1}{p^{ed}} \length_{S}\left(
  (J^{[p^{e}]}:J) \big/ \left( (J^{[p^{e}]}:J) \cap
    \n^{[p^{e}]}\right) \right) \, \, .
\]
\end{corollary}

\begin{proof}
The Corollary follows at once by taking $\sD = \sC^{R}$ in Theorem~\ref{thm:FedderComputationForFGradedSystems}.
\end{proof}

\begin{corollary}
  \label{cor:attightclosuretest}
 Suppose $(S, \n)$ is an $F$-finite regular local ring of dimension $n$. If $\a$ is a non-zero ideal and $t \in \R_{>0}$, then all ideals of $S$ are $\a^{t}$-tightly closed (see \cite{HaraYoshidaGeneralizationOfTightClosure}) if and only if
\[
s(S,\a^{t}) = \lim_{e \to \infty} \frac{1}{p^{ne}}\length_{S}\left( S \big/ \left( \n^{[p^{e}]} : \a^{\lceil t(p^{e}-1) \rceil}\right)\right) = \lim_{e \to \infty} \frac{1}{p^{ne}}\length_{S}\left(\a^{\lceil t(p^{e}-1) \rceil} \big/ \left( \a^{\lceil t(p^{e}-1) \rceil} \cap \n^{[p^{e}]}\right) \right) > 0 \, \, .
\]
\end{corollary}

\begin{proof}
  Using the results of \cite[Section 3.1]{BlickleSchwedeTakagiZhang}, the Corollary follows at once by taking the $F$-graded system from Example~\ref{ex.TwoExamplesOfFGraded}~(b).
\end{proof}

\begin{remark}
\label{rmk:prodofideals}
  More generally, if $\ba_1, \dots, \ba_s \subseteq S$ are non-zero
  ideals and $t_1, \dots, t_s \in \R_{\geq 0}$, then the sequence of
  ideals $\b_{e} = (\a_{1}^{\lceil t_{1}(p^{e}-1) \rceil}) \cdots
  (\a_{s}^{\lceil t_{s}(p^{e}-1) \rceil})$ forms an $F$-graded
  sequence. When $S$ is regular, the analogous
  statement to Corollary~\ref{cor:attightclosuretest} is valid for
  $s(S,\a_{1}^{t_{1}}\cdots \a_{s}^{t_{s}})$.
\end{remark}

\subsection{Examples}
\label{sec:some-examples}
\subsubsection{Divisor pairs and triples $(R, \Delta, \a^{t})$.}
\label{sec:divisor-pairs}
We have repeatedly seen above how Cartier subalgebras arise naturally in the context of a change of rings.  Furthermore, in Example \ref{ex.TwoExamplesOfFGraded}(b), we have shown how they can be used to study pairs $(R, \a^{t})$ where $\a \subseteq R$ is an ideal and $t \in \R_{\geq 0}$.  Let us now discuss how Cartier subalgebras can also be used to study the divisor pairs that appear throughout birational algebraic geometry. We refer the reader to \cite[Appendix B]{SchwedeTuckerTestIdealSurvey} for a brief algebraic summary of the divisor notation used below.

Suppose that $(R, \m)$ is an $F$-finite local normal domain and $\Delta$ is an effective $\R$-divisor on $\Spec(R)$.  For each $e \in \Z_{\geq 0}$, we have an inclusion $R \subseteq R(\lceil (p^{e} - 1) \Delta \rceil)$ (using $\Delta \geq 0$) where $R(\lceil (p^e - 1) \Delta \rceil)$ denotes the global sections of $\O_{\Spec (R)}(\lceil (p^e - 1)\Delta \rceil)$. In particular, we
may view
\[
\Hom_R(F^e_* R( \lceil (p^e - 1) \Delta \rceil), R) \subseteq \Hom_{R}(F^{e}_{*}R,R) = \sC_{e}^{R}
\]
and define
\[
\sC^{\Delta}_e = \Hom_R(F^e_* R( \lceil (p^e - 1) \Delta \rceil), R) \, \, .
\]
Again, it is straightforward to check that $\sC^{\Delta} = \oplus_{e \geq 0} \sC^{\Delta}_{e}$ is indeed a Cartier subalgebra.
If in addition we are given a non-zero ideal $\a \subseteq R$ and a coefficient $t \in \R_{\geq 0}$, we can incorporate the $F$-graded system in Example~\ref{ex.TwoExamplesOfFGraded}(b) as well.  Specifically, setting
\[
\sC^{\Delta, \ba^t}_e =  \sC^{\Delta}_e \cdot  F^e_* \ba^{\lceil t(p^e - 1) \rceil}
\]
gives rise to yet another Cartier subalgebra $\sC^{\Delta,\a^{t}} = \oplus_{e \geq 0}\sC^{\Delta,\a^{t}}_{e}$.  As before, we shall conflate $(R, \Delta)$ and $(R, \Delta, \a^{t})$ with corresponding pairs $(R, \sC^{\Delta})$ and $(R, \sC^{\Delta, \a^{t}})$, respectively, in the notation and terminology appearing throughout.

A number of variants on the Cartier subalgebra $\sC^{\Delta, \a^{t}}$ appear throughout the literature.  For example, one might consider replacing $\lceil (p^{e} - 1) \Delta \rceil$ with $\lceil p^{e}\Delta \rceil$ and $\a^{\lceil t(p^{e} -1)\rceil}$ with $\a^{\lceil tp^{e} \rceil}$ in the construction above.  The subsequent Cartier subalgebra may in fact be quite different; it often happens that the former is $F$-pure while the latter is not.  Nevertheless, it follows from the Lemma below that their $F$-signatures coincide.

\begin{lemma}
\label{lem:perturbalgebrasbyc}
Consider a Cartier subalgebra $\sD$ on an $F$-finite local ring $(R,\m)$ with dimension~$d$. Suppose that $\sE \subseteq \sD$ is another Cartier subalgebra with the property that there exists $c \in R$ not in any minimal prime satisfying $\sD_e \cdot F^{e}_{*}c \subseteq \sE_e$ for all $e > 0$.  Then
\[
s(R, \sD) = s(R, \sE)
\]
\end{lemma}
\begin{proof}
The inclusions $\sD_e \cdot F^{e}_{*}c \subseteq \sE_e \subseteq \sD_{e}$ immediately give
\[
I_{e}^{\sD} \subseteq I_{e}^{\sE} \subseteq \{ \, r \in R \; | \; \phi(F^{e}_{*}(cr)) \in \m \mbox{ for all } \phi \in \sD_{e}\, \} = (I_{e}^{\sD} : c )
\]
for all $e > 0$. Thus, we have
\[
\length_{R}(I_{e}^{\sE}/I_{e}^{\sD}) \leq \length_{R}\left( (I_{e}^{\sD} : c) \big/ I_{e}^{\sD}\right) = \length_{R} \left( R \big/ \langle I_{e}^{\sD},c \rangle\right)
\]
by noting that the kernel and cokernel  of the endomorphism $R/I_{e}^{\sD} \to[\cdot c] R/I_{e}^{\sD}$ have equal length (\textit{cf.} \cite[Page 8]{VraciuNewatTightClosure}).
As $\m^{[p^{e}]} \subseteq I_{e}^{\sD}$ and $\dim(R/\langle c \rangle) < d$, there exists a positive constant $C$ such that
\[
\length_{R}(I_{e}^{\sE}/I_{e}^{\sD})  \leq  \length_{R} \left( R \big/ \langle I_{e}^{\sD},c \rangle\right) \leq \length_{R}(R/\langle \m^{[p^{e}]},c\rangle) \leq C p^{e(d-1)}
\]
for all $e > 0$.  If we set $n = n_{\sD}n_{\sE}$, it follows that
\[
s(R, \sD) = \lim_{m \to \infty} \frac{1}{p^{nmd}}\length_{R}(R/I_{nm}^{\sD}) = \lim_{m \to \infty} \frac{1}{p^{nmd}}\length_{R}(R/I_{nm}^{\sE}) = s(R, \sE)
\]
as desired.
\end{proof}

\begin{remark}
  We caution the reader that it is not possible to replace the $F$-signature by the $F$-splitting ratio in Lemma~\ref{lem:perturbalgebrasbyc}.
\end{remark}

Lemma~\ref{lem:perturbalgebrasbyc} is particularly useful when computing the $F$-signature of the triples $(R, \Delta, \a^{t})$ considered above.  As mentioned previously, taking a non-zero $c \in \a^{\lceil t \rceil}$ with $\Div(c) \geq \Delta$ gives $c \a^{\lceil t(p^{e}-1) \rceil} \subseteq \a^{\lceil tp^{e} \rceil} \subseteq \a^{\lceil t(p^{e}-1)\rceil}$ and $R( \lceil (p^{e}-1)\Delta\rceil) \subseteq R( \lceil p^{e}\Delta\rceil) \subseteq c \cdot R( \lceil (p^{e}-1)\Delta\rceil)$,
so that we are free to consider the Cartier subalgebra
\[
\oplus_{e \geq 0} \Hom_{R}(F^{e}_{*}R( \lceil p^{e}\Delta\rceil),R) \cdot  F^e_* \ba^{\lceil tp^e \rceil}
\]
when computing $s(R, \Delta, \a^{t})$.  Alternatively,
recall that the integral closure of $\a$ with coefficient $\lambda \in \R_{\geq 0}$ is denoted $\overline{\a^{\lambda}}$, and we have $x \in \overline{\a^{\lambda}}$ if and only if $\nu(x) \geq \lambda \cdot \nu(\a)$ for all valuations $\nu \: \Frac(R)\setminus\{0 \} \to \Z$ with $\nu(R) \geq 0$. It is immediate that
$
\overline{\a^{\lambda}} = \bigcup_{\lambda< \mu \in \Q} \overline{\a^{\mu}} \, \, ,
$
and the Noetherian property of $R$ implies that this union stabilizes for some $\mu \in \Q$ (see \cite[Section 10.6]{HunekeSwansonIntegralClosure} for a detailed treatment in this setting).
The tight-closure {B}rian\c con-{S}koda theorem  (see \cite[Theorem 5.4]{HochsterHunekeTC1}) implies the existence of a non-zero $c \in R$ such that $c \cdot \overline{\ba^{t(p^e-1)}} \subseteq \ba^{\lceil t(p^e - 1) \rceil} \subseteq \overline{\ba^{t(p^e-1)}}$ for all $e \geq 0$, so that we might even consider the Cartier subalgebra
\[
\oplus_{e \geq 0} \Hom_{R}(F^{e}_{*}R( \lceil (p^{e}-1) \Delta\rceil),R) \cdot  F^e_* \left( \, \overline{\ba^{t(p^e - 1)}} \, \right)
\]
when computing $s(R,\Delta,\a^{t})$.

We conclude this subsection with the computation of the $F$-signature of a fundamentally important divisor pair.

\begin{example}
\label{ex:SNCDivisor}
Suppose that $R = k\llbracket x_1, \dots, x_n \rrbracket $ where $k =
k^{p}$ is a perfect field.  Consider the $\bR$-divisor $\Delta = t_1
\Div(x_1) + \dots + t_n \Div(x_n)$ where $t_i \in [0, 1)$.  It is
straightforward to see that $ s(R, \Delta) = s(R,\langle x_{1}
\rangle^{t_{1}}\cdots \langle x_{n} \rangle^{t_{n}})$.  By
Remark~\ref{rmk:prodofideals}, we have
\[
s(R, \Delta) = \lim_{e \to \infty} \frac{1}{p^{en}}\length_{R}\left( R \big/ (\langle
  x_{1}^{p^{e}}, \ldots, x_{n}^{p^{e}} \rangle : x_{1}^{\lceil
    t_{1}(p^{e}-1)\rceil} \cdots x_{n}^{\lceil
    t_{n}(p^{e}-1)\rceil})\right) \, \, .
\]
Since $p^{e} - \lceil t_{i}(p^{e}-1) \rceil = \lfloor (1-t_{i})p^{e} +
t_{i}\rfloor$, we have
\[
(\langle
  x_{1}^{p^{e}}, \ldots, x_{n}^{p^{e}} \rangle : x_{1}^{\lceil
    t_{1}(p^{e}-1)\rceil} \cdots x_{n}^{\lceil
    t_{n}(p^{e}-1)\rceil}) = \langle x_{1}^{\lfloor (1-t_{1})p^{e} +
t_{1}\rfloor}, \ldots, x_{n}^{\lfloor (1-t_{n})p^{e} +
t_{n}\rfloor}\rangle \, \, .
\]
Thus, we compute
\[
\length_{R}\left( R \big/ (\langle
  x_{1}^{p^{e}}, \ldots, x_{n}^{p^{e}} \rangle : x_{1}^{\lceil
    t_{1}(p^{e}-1)\rceil} \cdots x_{n}^{\lceil
    t_{n}(p^{e}-1)\rceil})\right) = \prod_{i = 1}^{n} \lfloor (1-t_{i})p^{e} +
t_{i}\rfloor
\]
in order to conclude
\[
s(R, {\Delta}) = (1 - t_1) \cdots (1 - t_n).
\]
\end{example}

\subsubsection{Whitney's Umbrella}
\label{sec:whitneys-umbrella}
We show in this example that the $F$-splitting ratio of a local ring
may be strictly smaller than the $F$-signature of the quotient by the $F$-splitting prime.
Assume $p \neq 2$ and consider the two dimensional local ring
\[
R =
\mathbb{F}_{p}\llbracket x,y,z \rrbracket /\langle x^{2} - y^{2}z \rangle \, \, .
\]
If $S = \mathbb{F}_{p}\llbracket x,y,z \rrbracket $ and $\overline{R}
= R / P_{R}$, we will show $r_{F}(R) = 1/2$ while $s(\overline{R}) = 1$.
Set $\Phi \: F_{*}S \to S$ to denote the homomorphism
mapping $F_{*}(x^{p-1}y^{p-1}z^{p-1}) \mapsto 1$, and all of the other
monomials in the free basis $\{F_{*}(x^{i}y^{j}z^{k})\}_{0\leq i,j,k
  \leq p^{e}-1}$ of the free $S$-module $F_{*}S$ to zero.  If $f =
x^{2}-y^{2}z$, then $\Phi(F_{*}f^{p-1} \cdot \blank)$ induces a
homomorphism $\phi \in \Hom_{R}(F_{*}R,R)$ such that $\sC^{R}
=\sC^{\phi}$.  In other words, $\phi^{e}$ is an $F_{*}^{e}R$-module generator for $\Hom_{R}(F^{e}_{*}R,R)$ for all $e > 0$.  Since
\[
f^{p-1} = {p-1 \choose 2} x^{p-1}y^{p-1}z^{(p-1)/2} \mod \langle x^{p}, y^{p} \rangle
\]
and ${p-1 \choose 2} \neq 0 \mod p$, it is easy to see that $\langle
x, y \rangle \subseteq R$ is $\phi$-compatible thus inducing a map
$\overline{\phi} \in \Hom_{\overline{R}}(F_{*}\overline{R},\overline{R})$.  Now, the (total) Cartier
algebra of the ring $\overline{R} = k\llbracket x,y,z \rrbracket
/\langle x, y \rangle \cong k\llbracket z \rrbracket $ is generated by
the map which sends $F_{*}z^{p-1}$ to $1$ and the lower order monomials
$F_{*}z^i$ to zero for $0 \leq i < p-1$.  It is easy to see that
$\overline{\phi}$ sends $F_{*}z^{(p-1)/2}$ to $1$ and all other
$F_{*}z^i$ to zero for $0 \leq i \leq p-1$, $i \neq {p-1 \over 2}$.
Thus, the Cartier subalgebra $\sC^{1/2\Div(z)} \subseteq
\sC^{k\llbracket z \rrbracket }$ coincides with the Cartier subalgebra
generated by $\overline{\phi}$.
Since $(k\llbracket z \rrbracket, {1 \over 2} \Div(z))$ is
$F$-regular, it immediately follows that $P_{R} = \langle x, y \rangle
\cdot R$ is the splitting prime by Lemma \ref{lem.BijectionBetweenCompatibleIdeals}, \cf \cite{SchwedeFAdjunction}.
Furthermore, from example
Example~\ref{ex:SNCDivisor}, we see that
\[
r_{F}(R) = r_{F}(R, \sC^{\phi}) = r_{F}(\overline{R}, \overline{\phi}) = r_{F}(k\llbracket z \rrbracket,
{1 \over 2} \Div(z)) = s (k\llbracket z \rrbracket,
{1 \over 2} \Div(z)) = \frac{1}{2} \, \, .
\]
In particular, we have $r_{F}(R) = 1/2 < 1 = s(k\llbracket z \rrbracket)= s(\overline{R})$.

\subsubsection{A ramified cover of the cusp in characteristic three.}
\label{sec:ramified-cover-cusp}
Consider the ring\footnote{
Using Macaulay2 \cite{M2}, one can check $R \cong \bF_3\llbracket
x,y,z^2 - x^3 + y^2, z(z^2 - x^3 + y^2)\rrbracket \subseteq
\bF_3\llbracket x,y,z\rrbracket$.
}
\[
R = \bF_3\llbracket a,b,c,d \rrbracket/\langle a^3 c^2 -b^2c^2 + c^3 -
d^2 \rangle
\]
 so that $R = S / \langle h \rangle$ where $S = \mathbb{F}_{3}\llbracket a,b,c,d \rrbracket$ and $h =
 a^3 c^2 -b^2c^2 + c^3 - d^2$.  Since
\[
  h^2 = a^6 c^4 +a^3b^2 c^4 + b^4 c^4 - a^3 c^5 + b^2c^5 + a^3c^2d^2
  + c^6 - b^2c^2d^2 + c^3d^2 + d^4
\not\in (a^3, b^3, c^3, d^3) \, \, ,
\]
we see that $R$ is $F$-pure using Fedder's criterion
\cite{FedderFPureRat}.  As in the previous example,
$F_{*}h^{2}$ corresponds to a
generator $\phi$ of $\Hom_R(F_* R, R)$ via \cite{FedderFPureRat}, so
that $\sC^{R} = \sC^{\phi}$.  It
is easy to see that $P = \langle c, d \rangle \subseteq R$ is
$\sC^{R}$-compatible, and the induced map $\overline{\phi} \in
\Hom_{\overline{R}}(F_{*}\overline{R},\overline{R})$ on $\overline{R}
= R/P$ corresponds on $\overline{R} = \bF_{3}\llbracket a,b \rrbracket$ to
pre-multiplying a generator of
$\Hom_{\overline{R}}(F_*\overline{R},\overline{R})$ by $F_{*}(a^3 -
b^2)$ (this is the coefficient of $c^2d^2$ in the expression for $h^{2}$
above).  Thus $\sC^{\overline{\phi}} = \sC^{{1 \over 2} \Div(a^3 -
  b^2)}$.  Now, by \cite[Example 4.3]{MustataTakagiWatanabeFThresholdsAndBernsteinSato}, \cf \cite{TakagiWatanabeFPureThresh, DanielThesis}, the $F$-pure threshold of $a^3 - b^2$ in characteristic $3$ is ${2 \over 3}$ and so
it follows that $(\bF_{3}\llbracket a,b\rrbracket, {1 \over 2}
\Div(a^3 - b^2))$ is (strongly) $F$-regular.  We conclude that $P$ is
the splitting prime of $R$, and $r_{F}(R) = r_{F}(R, \sC^{\phi}) =
r_{F}(\overline{R}, \sC^{\overline{\phi}}) =s(\bF_{3}\llbracket a,b\rrbracket, {1 \over 2}
\Div(a^3 - b^2))$.

Let us now compute $s(\bF_{3}\llbracket a,b\rrbracket, {1 \over 2}
\Div(a^3 - b^2)) = 1/6$ by taking a pair of finite covers.
First set $A = \bF_{3}\llbracket a,b \rrbracket \subseteq
\bF_3\llbracket a, d = b^{1/3}\rrbracket  = \bF_3\llbracket a, d \rrbracket = B$, and notice that $a^3 - b^2 = a^3 - d^6 = (a - d^2)^3$.  Since $B$ is a free $A$-module of rank $3$, we see that
\[
 \length_A\left( {A \over  \langle a^{3^e}, b^{3^e} \rangle : (a^3 -
     b^2)^{\lceil \frac{3^{e}}{2} \rceil}}\right) = {1 \over 3} \length_B\left({B \over \langle a^{3^e}, d^{3^{e+1}} \rangle : (a - d^2)^{3(3^e + 1)/2} }\right).
\]
Now consider $B \subseteq \bF_3\llbracket a, d, c = (a - d^2)^{1/2}
\rrbracket = \bF_e\llbracket d,c \rrbracket = C$ and observe that
\[
\length_A\left( {A \over  \langle a^{3^e}, b^{3^e} \rangle : (a^3 - b^2)^{\lceil \frac{3^{e}}{2} \rceil}}\right) = {1 \over 6} \length_C\left({C \over \langle (c^{2}+d^{2})^{3^e} , d^{3^{e+1}} \rangle : c^{3^{e+1} +3}}\right).
\]
Writing $\n = \langle  (c^{2}+d^{2}) ,
d^{3}\rangle$ and thus $ \n^{[3^{e}]} = \langle (c^{2}+d^{2})^{3^e}
, d^{3^{e+1}} \rangle$ for ease of notation, it is straightforward to check $\n : c^{3} = \langle c, d \rangle$.
Now, note that replacing $c^{3^{e+1}+3}$ by
$c^{3^{e+1}}$ in the formula above gives
\[
\n^{[3^{e}]}:c^{3^{e+1}} = (\n : c^{3})^{[3^{e}]} =  \langle
c^{3^{e}}, d^{3^{e}} \rangle \, \, .
\]
Thus, we compute
\[
\n^{[3^{e}]}:c^{3^{e+1}+3} = (\n^{[3^{e}]}:c^{3^{e+1}}):c^{3} = \langle c^{3^{e}}, d^{3^{e}} \rangle : c^{3} =
\langle c^{3^{e} - 3}, d^{3^{e}} \rangle
\]
so that we have
\[
\length_A\left( {A \over  \langle a^{3^e}, b^{3^e} \rangle : (a^3 -
    b^2)^{\lceil \frac{3^{e}}{2} \rceil}}\right) = {1 \over 6}\left(
  (3^{e}-3)(3^{e})\right) = {1 \over 6}\left( 3^{2e} - 3^{e+1} \right)
\, \, .
\]
Now, using Corollary~\ref{cor:attightclosuretest} together with the
discussion following
Lemma~\ref{lem:perturbalgebrasbyc}, we conclude
that
\[
r_{F}(R) = s\left( A,{1 \over 2}
\Div(a^3 - b^2)\right) = \lim_{e \to \infty}
\frac{1}{3^{2e}}\length_A\left( {A \over  \langle a^{3^e}, b^{3^e}
    \rangle : (a^3 - b^2)^{\lceil \frac{3^{e}}{2} \rceil}}\right) = \frac{1}{6} \, \, .
\]

\subsubsection{Monomial ideals.}
\label{sec:monomial-ideals}
\renewcommand{\bP}{{\bf P}}
In this section, we consider $S = k \llbracket x_{1}, \ldots, x_{n} \rrbracket$ where $k$ is an $F$-finite field of characteristic $p >0$.  If we have $u = (u_{1}, \ldots, u_{n}) \in \Z^{n}$, then $x^{u} = x_{1}^{u_{1}}\cdots x_{n}^{u_{n}}$ will denote the corresponding monomial in the variables $x_{1}, \ldots, x_{n}$.  If $\a \subseteq S$ is a monomial ideal, \textit{i.e.} is generated by monomials in $x_{1}, \ldots, x_{n}$, we can consider the Newton polyhedron $\bP_{\a} \subseteq \R^{n}$ of $\a$ given by the closed convex hull in $\R^{n}$ of the set
$
\{ \, u \in \Z^{ n} \; | \; x^{u} \in \a \}
$
of lattice points representing the monomials in $\a$.

It is well known that the integral closure $\overline{\a}$ is also a
monomial ideal and is determined by $\bP_{\a}$, so that $x^{u} \in
\overline{\a}$ if and only if $u \in \bP_{\a} \cap \Z^{ n}$.  If
$\lambda \in \R_{> 0}$, this is readily generalized to the statement
that $\overline{\a^{\lambda}}$ is a monomial ideal determined by
$\lambda \bP_{\a} = \{ \, \lambda x \; | \; x \in \bP_{\a} \, \}$, so
that $x^{u} \in \overline{\a^{\lambda}}$ if and only if $u \in \lambda
\bP_{\a} \cap \Z^{ n}$.  Indeed, the Rees valuations of $\a$ are the
monomial valuations obtained from the bounding hyperplanes of $\bP_{\a}$ \cite[Proposition
10.3.5]{HunekeSwansonIntegralClosure} (\textit{cf.} \cite[Section 11.3]{CoxLittleSchenckToric}) and
determine membership in
$\overline{\a^{\lambda}}$ \cite[Proposition
10.6.5]{HunekeSwansonIntegralClosure}.

\begin{theorem}
  \label{thm:monomialidealcalc}
Let
$S = k \llbracket x_{1}, \ldots, x_{n} \rrbracket$ where $k$ is an $F$-finite field of characteristic $p >0$.  Suppose $\a \subseteq S$ is a monomial ideal with Newton polyhedron $\bP_{\a} \subseteq \R^{ n}$ and $t \in \R_{> 0}$.  Then the $F$-signature of $(S, \a^{t})$
\[
s(S, \a^{t}) = \vol_{\R^{ n}}(t\bP_{\a} \cap [0,1]^{ n})
\]
equals the Euclidean volume of the intersection of $t \bP_{\a}$ with the unit cube $[0,1]^{ n} \subseteq \R^{n}$.
\end{theorem}

\begin{proof}
  Set $\n = \langle x_{1}, \ldots, x_{n} \rangle$.  By Lemma~\ref{lem:perturbalgebrasbyc} and Lemma~\ref{lem:compsinregrings}, we have that
\[
s(S, \a^{t}) = \lim_{e \to \infty} \frac{1}{p^{ne}} \length_{S} \left( \frac{\overline{\a^{tp^{e}}}}{\overline{\a^{tp^{e}}}\cap \n^{[p^{e}]}} \right) \, \, .
\]
As noted above, $\overline{\a^{tp^{e}}}$ is a monomial ideal determined by the property $x^{u} \in \overline{\a^{tp^{e}}}$ if and only if $u \in tp^{e}\bP_{\a} \cap \Z^{n}$.  Likewise, $\overline{\n^{[p^{e}]}}$ is a monomial ideal, and $x^{u} \not\in \n^{[p^{e}]}$ if and only if $u \in [0,p^{e}]^{n} \cap \Z^{n}$.  In particular, since the lengths appearing in the limit above are a quotient of monomial ideals, we need only count monomials to compute them.  We have
\[
\begin{array}{rcl}
\displaystyle \length_{S} \left( \frac{\overline{\a^{tp^{e}}}}{\overline{\a^{tp^{e}}}\cap \n^{[p^{e}]}} \right) &=& \# \{ \, u \in \Z^{n} \; | \; x^{u} \in \overline{\a^{tp^{e}}} \setminus \n^{[p^{e}]} \, \} \\
\phantom{\length_{S} \left( \frac{\overline{\a^{tp^{e}}}}{\overline{\a^{tp^{e}}}\cap \n^{[p^{e}]}} \right)} & = & \# \left( \Z^{n} \cap (tp^{e}\bP_{\a}) \cap [0,p^{e}]^{n} \right) \\
\phantom{\length_{S} \left( \frac{\overline{\a^{tp^{e}}}}{\overline{\a^{tp^{e}}}\cap \n^{[p^{e}]}} \right)} & = & \# \left( \Z^{n} \cap \left( p^{e} \cdot \left( t\bP_{\a} \cap [0,1]^{n} \right)\right)\right) \, \, .
\end{array}
\]
The result now follows from the following elementary identity: if
\renewcommand{\bQ}{\mbox{{\bfseries Q}}}
$\bQ \subseteq \R^{n}$ is any polytope, then $\displaystyle \vol_{\R^{n}}(\bQ) = \lim_{q \to \infty} \frac{\#(q\bQ)}{q^{n}}$.  See \cite[Page 111]{FultonIntroToric} for further discussion.
\end{proof}

\begin{remark}
Note, in particular, that the formula in Theorem~\ref{thm:monomialidealcalc} above is independent of both the coefficient field $k$ and even the characteristic $p > 0$.

\end{remark}

\begin{example}
\label{ex:monomialcuspidealexample}
Consider now the ideal $\a = \langle x^{3}, y^{2} \rangle$ inside $k \llbracket x,y \rrbracket$ where $k$ is an $F$-finite field of characteristic $p >0$.  If $t \in \R_{\geq 0}$, the intersection $t \bP_{\a} \cap [0,1]^{n}$ is non-trivial only if $t \leq 5/6$, in which case it takes on one of the three forms indicated in the diagrams below.
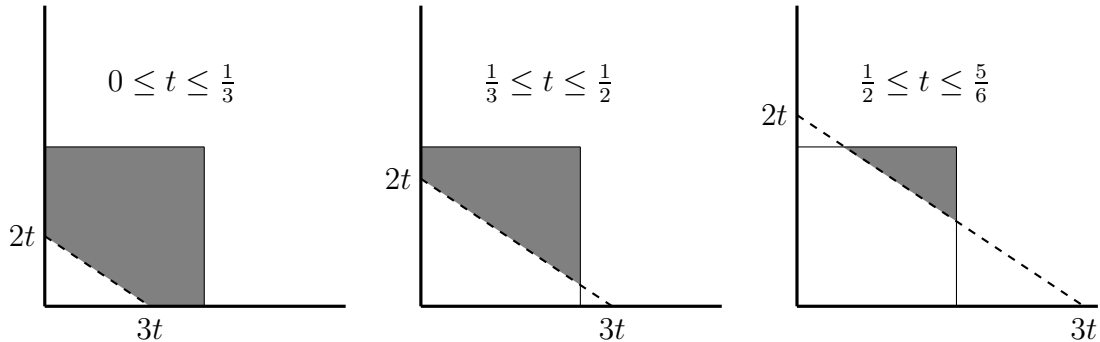
\begin{center}
\begin{tikzpicture}

\fill[gray] (0, 2*2.12*0.22) -- (3*2.12*0.22, 0) -- (2.12, 0) -- (2.12, 2.12) -- (0, 2.12) -- cycle;
\begin{scope}[very thick]
\draw(0,0)--(4,0);
\draw(0,0)--(0,4);
\end{scope}
\begin{scope}
\draw(0, 0)--(2.12, 0);
\draw(0, 0)--(0, 2.12);
\draw(2.12, 0)--(2.12,2.12);
\draw(0,2.12)--(2.12,2.12);
\end{scope}
\begin{scope}[dashed, thick]
\draw(0, 2*2.12*0.22)--(3*2.12*0.22, 0);
\end{scope}
\draw (0-0.3, 2*2.12*0.22) node{$2t$};
\draw (3*2.12*0.22, 0-0.3) node{$3t$};
\draw (0.8*2.12, 1.4*2.12) node{$0 \leq t \leq {1 \over 3}$};

\fill[gray] (5, 2*2.12*0.4) -- (5+2.12, -0.6666*2.12 +2*2.12*0.4) -- (5+2.12, 2.12) -- (5+0, 2.12) -- cycle;
\begin{scope}[very thick]
\draw(5,0)--(9,0);
\draw(5,0)--(5,4);
\end{scope}
\begin{scope}
\draw(5, 0)--(5+2.12, 0);
\draw(5, 0)--(5, 2.12);
\draw(5+2.12, 0)--(5+2.12,2.12);
\draw(5,2.12)--(5+2.12,2.12);
\end{scope}
\begin{scope}[dashed, thick]
\draw(5, 2*2.12*0.4)--(5+3*2.12*0.4, 0);
\end{scope}
\draw (5-0.3, 2*2.12*0.4) node{$2t$};
\draw (5+3*2.12*0.4, 0-0.3) node{$3t$};
\draw (5+0.8*2.12, 1.4*2.12) node{${1 \over 3} \leq t \leq {1 \over 2}$};

\fill[gray] (10- 2.12*1.5 + 2*2.12*0.6*1.49, 2.12) -- (10+2.12, 2.12) -- (10+2.12, 2*2.12*0.6-0.6666*2.12) -- cycle;
\begin{scope}[very thick]
\draw(10,0)--(14,0);
\draw(10,0)--(10,4);
\end{scope}
\begin{scope}
\draw(10, 0)--(10+2.12, 0);
\draw(10, 0)--(10, 2.12);
\draw(10+2.12, 0)--(10+2.12,2.12);
\draw(10,2.12)--(10+2.12,2.12);
\end{scope}
\begin{scope}[dashed, thick]
\draw(10, 2*2.12*0.6)--(10+3*2.12*0.6, 0);
\end{scope}
\draw (10-0.3, 2*2.12*0.6) node{$2t$};
\draw (10+3*2.12*0.6, -0.3) node{$3t$};
\draw (10+0.8*2.12, 1.4*2.12) node{${1 \over 2} \leq t \leq {5 \over 6}$};

\end{tikzpicture}
\end{center}
Using Theorem~\ref{thm:monomialidealcalc}, one readily computes
\[
s(k\llbracket x,y \rrbracket, \langle x^{3}, y^{2} \rangle^{t} ) = \left\{
  \begin{array}{c@{\qquad}c}
    1 - 3t^{2} & 0 \leq t \leq \frac{1}{3} \vspace{1mm} \\
    \frac{4}{3} - 2t & \frac{1}{3} \leq t \leq \frac{1}{2} \vspace{1mm}\\
\frac{25}{12} - 5t + 3t^{2} & \frac{1}{2} \leq t \leq \frac{5}{6}
  \end{array}
\right. \, \, .
\]

\end{example}

\bibliographystyle{skalpha}
\bibliography{CommonBib}

\def\cprime{$'$} \def\cprime{$'$}
  \def\cfudot#1{\ifmmode\setbox7\hbox{$\accent"5E#1$}\else
  \setbox7\hbox{\accent"5E#1}\penalty 10000\relax\fi\raise 1\ht7
  \hbox{\raise.1ex\hbox to 1\wd7{\hss.\hss}}\penalty 10000 \hskip-1\wd7\penalty
  10000\box7} \def\cfudot#1{\ifmmode\setbox7\hbox{$\accent"5E#1$}\else
  \setbox7\hbox{\accent"5E#1}\penalty 10000\relax\fi\raise 1\ht7
  \hbox{\raise.1ex\hbox to 1\wd7{\hss.\hss}}\penalty 10000 \hskip-1\wd7\penalty
  10000\box7} \def\cfudot#1{\ifmmode\setbox7\hbox{$\accent"5E#1$}\else
  \setbox7\hbox{\accent"5E#1}\penalty 10000\relax\fi\raise 1\ht7
  \hbox{\raise.1ex\hbox to 1\wd7{\hss.\hss}}\penalty 10000 \hskip-1\wd7\penalty
  10000\box7}
\providecommand{\bysame}{\leavevmode\hbox to3em{\hrulefill}\thinspace}
\providecommand{\MR}{\relax\ifhmode\unskip\space\fi MR}
\providecommand{\MRhref}[2]{%
  \href{http://www.ams.org/mathscinet-getitem?mr=#1}{#2}
}
\providecommand{\href}[2]{#2}
\begin{thebibliography}{MTW05}

\bibitem[Abe01]{AberbachExtensionofFRegularRingsbyFlatMaps}
{\sc I.~M. Aberbach}: \emph{Extension of weakly and strongly {F}-regular rings
  by flat maps}, J. Algebra \textbf{241} (2001), no.~2, 799--807.
  {\sf\scriptsize 1843326 (2002f:13008)}

\bibitem[AE05]{AberbachEnescuStructureOfFPure}
{\sc I.~M. Aberbach and F.~Enescu}: \emph{The structure of {$F$}-pure rings},
  Math. Z. \textbf{250} (2005), no.~4, 791--806. {\sf\scriptsize MR2180375}

\bibitem[AL03]{AberbachLeuschke}
{\sc I.~M. Aberbach and G.~J. Leuschke}: \emph{The {$F$}-signature and strong
  {$F$}-regularity}, Math. Res. Lett. \textbf{10} (2003), no.~1, 51--56.
  {\sf\scriptsize MR1960123 (2004b:13003)}

\bibitem[Bli09]{BlickleTestIdealsViaAlgebras}
{\sc M.~Blickle}: \emph{Test ideals via algebras of $p^{-e}$-liner maps},
  arXiv:0912.2255, to appear in J. Algebraic Geom.

\bibitem[BSTZ10]{BlickleSchwedeTakagiZhang}
{\sc M.~Blickle, K.~Schwede, S.~Takagi, and W.~Zhang}: \emph{Discreteness and
  rationality of {$F$}-jumping numbers on singular varieties}, Math. Ann.
  \textbf{347} (2010), no.~4, 917--949. {\sf\scriptsize 2658149}

\bibitem[BK05]{BrionKumarFrobeniusSplitting}
{\sc M.~Brion and S.~Kumar}: \emph{Frobenius splitting methods in geometry and
  representation theory}, Progress in Mathematics, vol. 231, Birkh\"auser
  Boston Inc., Boston, MA, 2005. {\sf\scriptsize MR2107324 (2005k:14104)}

\bibitem[CLS11]{CoxLittleSchenckToric}
{\sc D.~Cox, J.~Little, and H.~Schenck}: \emph{Toric varieties}, Graduate
  Studies in Mathematics, vol. 124, American Mathematical Society, Providence,
  RI, 2011.

\bibitem[ELS01]{EinLazSmithSymbolic}
{\sc L.~Ein, R.~Lazarsfeld, and K.~E. Smith}: \emph{Uniform bounds and symbolic
  powers on smooth varieties}, Invent. Math. \textbf{144} (2001), no.~2,
  241--252. {\sf\scriptsize MR1826369 (2002b:13001)}

\bibitem[EY11]{EnescuYaoLowerSemicontinuity}
{\sc F.~Enescu and Y.~Yao}: \emph{The lower semicontinuity of the {F}robenius
  splitting numbers}, Math. Proc. Cambridge Philos. Soc. \textbf{150} (2011),
  no.~1, 35--46. {\sf\scriptsize 2739072}

\bibitem[Fed83]{FedderFPureRat}
{\sc R.~Fedder}: \emph{{$F$}-purity and rational singularity}, Trans. Amer.
  Math. Soc. \textbf{278} (1983), no.~2, 461--480. {\sf\scriptsize MR701505
  (84h:13031)}

\bibitem[Ful93]{FultonIntroToric}
{\sc W.~Fulton}: \emph{Introduction to toric varieties}, Annals of Mathematics
  Studies, vol. 131, Princeton University Press, Princeton, NJ, 1993, The
  William H. Roever Lectures in Geometry. {\sf\scriptsize 1234037 (94g:14028)}

\bibitem[GS]{M2}
{\sc D.~R. Grayson and M.~E. Stillman}: \emph{Macaulay2, a software system for
  research in algebraic geometry}.

\bibitem[HW02]{HaraWatanabeFRegFPure}
{\sc N.~Hara and K.-I. Watanabe}: \emph{F-regular and {F}-pure rings vs. log
  terminal and log canonical singularities}, J. Algebraic Geom. \textbf{11}
  (2002), no.~2, 363--392. {\sf\scriptsize MR1874118 (2002k:13009)}

\bibitem[HY03]{HaraYoshidaGeneralizationOfTightClosure}
{\sc N.~Hara and K.-I. Yoshida}: \emph{A generalization of tight closure and
  multiplier ideals}, Trans. Amer. Math. Soc. \textbf{355} (2003), no.~8,
  3143--3174 (electronic). {\sf\scriptsize MR1974679 (2004i:13003)}

\bibitem[Her11]{DanielThesis}
{\sc D.~Hern\'andez}: \emph{{$F$}-purity for hypersurfaces}, Ph.D. thesis,
  University of Michigan, 2011.

\bibitem[HH89]{HochsterHunekeTightClosureAndStrongFRegularity}
{\sc M.~Hochster and C.~Huneke}: \emph{Tight closure and strong
  {$F$}-regularity}, M\'em. Soc. Math. France (N.S.) (1989), no.~38, 119--133,
  Colloque en l'honneur de Pierre Samuel (Orsay, 1987). {\sf\scriptsize
  MR1044348 (91i:13025)}

\bibitem[HH90]{HochsterHunekeTC1}
{\sc M.~Hochster and C.~Huneke}: \emph{Tight closure, invariant theory, and the
  {B}rian\c con-{S}koda theorem}, J. Amer. Math. Soc. \textbf{3} (1990), no.~1,
  31--116. {\sf\scriptsize MR1017784 (91g:13010)}

\bibitem[HH91]{HochsterHunekeTightClosureElementsSmallOrder}
{\sc M.~Hochster and C.~Huneke}: \emph{Tight closure and elements of small
  order in integral extensions}, J. Pure Appl. Algebra \textbf{71} (1991),
  no.~2-3, 233--247. {\sf\scriptsize 1117636 (92i:13002)}

\bibitem[HR74]{HochsterRobertsRingsOfInvariants}
{\sc M.~Hochster and J.~L. Roberts}: \emph{Rings of invariants of reductive
  groups acting on regular rings are {C}ohen-{M}acaulay}, Advances in Math.
  \textbf{13} (1974), 115--175. {\sf\scriptsize MR0347810 (50 \#311)}

\bibitem[HR76]{HochsterRobertsFrobeniusLocalCohomology}
{\sc M.~Hochster and J.~L. Roberts}: \emph{The purity of the {F}robenius and
  local cohomology}, Advances in Math. \textbf{21} (1976), no.~2, 117--172.
  {\sf\scriptsize MR0417172 (54 \#5230)}

\bibitem[Hun96]{HunekeTightClosureBook}
{\sc C.~Huneke}: \emph{Tight closure and its applications}, CBMS Regional
  Conference Series in Mathematics, vol.~88, Published for the Conference Board
  of the Mathematical Sciences, Washington, DC, 1996, With an appendix by
  Melvin Hochster. {\sf\scriptsize MR1377268 (96m:13001)}

\bibitem[HL02]{HunekeLeuschkeTwoTheoremsAboutMaximal}
{\sc C.~Huneke and G.~J. Leuschke}: \emph{Two theorems about maximal
  {C}ohen-{M}acaulay modules}, Math. Ann. \textbf{324} (2002), no.~2, 391--404.
  {\sf\scriptsize MR1933863 (2003j:13011)}

\bibitem[HS06]{HunekeSwansonIntegralClosure}
{\sc C.~Huneke and I.~Swanson}: \emph{Integral closure of ideals, rings, and
  modules}, London Mathematical Society Lecture Note Series, vol. 336,
  Cambridge University Press, Cambridge, 2006. {\sf\scriptsize MR2266432
  (2008m:13013)}

\bibitem[Kat10]{KatzmanANonFinitelyGeneratedAlgebra}
{\sc M.~Katzman}: \emph{A non-finitely generated algebra of {F}robenius maps},
  Proc. Amer. Math. Soc. \textbf{138} (2010), no.~7, 2381--2383.
  {\sf\scriptsize 2607867}

\bibitem[Kun69]{KunzCharacterizationsOfRegularLocalRings}
{\sc E.~Kunz}: \emph{Characterizations of regular local rings for
  characteristic {$p$}}, Amer. J. Math. \textbf{91} (1969), 772--784.
  {\sf\scriptsize MR0252389 (40 \#5609)}

\bibitem[Kun76]{KunzOnNoetherianRingsOfCharP}
{\sc E.~Kunz}: \emph{On {N}oetherian rings of characteristic {$p$}}, Amer. J.
  Math. \textbf{98} (1976), no.~4, 999--1013. {\sf\scriptsize MR0432625 (55
  \#5612)}

\bibitem[LS01]{LyubeznikSmithCommutationOfTestIdealWithLocalization}
{\sc G.~Lyubeznik and K.~E. Smith}: \emph{On the commutation of the test ideal
  with localization and completion}, Trans. Amer. Math. Soc. \textbf{353}
  (2001), no.~8, 3149--3180 (electronic). {\sf\scriptsize MR1828602
  (2002f:13010)}

\bibitem[Mat80]{MatsumuraCommutativeAlgebra}
{\sc H.~Matsumura}: \emph{Commutative algebra}, second ed., Mathematics Lecture
  Note Series, vol.~56, Benjamin/Cummings Publishing Co., Inc., Reading, Mass.,
  1980. {\sf\scriptsize MR575344 (82i:13003)}

\bibitem[MBZ11]{AlvarezmontanerBoixZarzuelaAlgebrasOfStanleyReisnerRings}
{\sc J.~A. Montaner, A.~F. Boix, and S.~Zarzuela}: \emph{Frobenius and cartier
  algebras of stanley-reisner rings}, arXiv:1106.5686.

\bibitem[MTW05]{MustataTakagiWatanabeFThresholdsAndBernsteinSato}
{\sc M.~Musta{\c{t}}{\v{a}}, S.~Takagi, and K.-i. Watanabe}: \emph{F-thresholds
  and {B}ernstein-{S}ato polynomials}, European Congress of Mathematics, Eur.
  Math. Soc., Z\"urich, 2005, pp.~341--364. {\sf\scriptsize MR2185754
  (2007b:13010)}

\bibitem[Ree88]{ReesAsymptoticTheoryofIdeals}
{\sc D.~Rees}: \emph{Lectures on the asymptotic theory of ideals}, London
  Mathematical Society Lecture Note Series, vol. 113, Cambridge University
  Press, Cambridge, 1988. {\sf\scriptsize 988639 (90d:13012)}

\bibitem[Sch08]{SchwedeSharpTestElements}
{\sc K.~Schwede}: \emph{Generalized test ideals, sharp {$F$}-purity, and sharp
  test elements}, Math. Res. Lett. \textbf{15} (2008), no.~6, 1251--1261.
  {\sf\scriptsize MR2470398}

\bibitem[Sch09]{SchwedeFAdjunction}
{\sc K.~Schwede}: \emph{{$F$}-adjunction}, Algebra Number Theory \textbf{3}
  (2009), no.~8, 907--950.

\bibitem[Sch10a]{SchwedeCentersOfFPurity}
{\sc K.~Schwede}: \emph{Centers of {$F$}-purity}, Math. Z. \textbf{265} (2010),
  no.~3, 687--714. {\sf\scriptsize 2644316}

\bibitem[Sch10b]{SchwedeBetterFPureFRegular}
{\sc K.~Schwede}: \emph{A refinement of sharply {$F$}-pure and strongly
  {$F$}-regular pairs}, J. Commut. Algebra \textbf{2} (2010), no.~1, 91--109.
  {\sf\scriptsize 2607103 (2011c:13007)}

\bibitem[Sch11]{SchwedeTestIdealsInNonQGor}
{\sc K.~Schwede}: \emph{Test ideals in non-{$\Bbb{Q}$}-{G}orenstein rings},
  Trans. Amer. Math. Soc. \textbf{363} (2011), no.~11, 5925--5941.
  {\sf\scriptsize 2817415 (2012c:13011)}

\bibitem[SS10]{SchwedeSmithLogFanoVsGloballyFRegular}
{\sc K.~Schwede and K.~E. Smith}: \emph{Globally {$F$}-regular and log {F}ano
  varieties}, Adv. Math. \textbf{224} (2010), no.~3, 863--894. {\sf\scriptsize
  2628797 (2011e:14076)}

\bibitem[ST12]{SchwedeTuckerTestIdealSurvey}
{\sc K.~Schwede and K.~Tucker}: \emph{A survey of test ideals}, Progress in
  commutative algebra 2, Walter de Gruyter, Berlin, 2012, pp.~39--99.
  {\sf\scriptsize 2932591}

\bibitem[SVdB97]{SmithVanDenBerghSimplicityOfDiff}
{\sc K.~E. Smith and M.~Van~den Bergh}: \emph{Simplicity of rings of
  differential operators in prime characteristic}, Proc. London Math. Soc. (3)
  \textbf{75} (1997), no.~1, 32--62. {\sf\scriptsize MR1444312 (98d:16039)}

\bibitem[TW04]{TakagiWatanabeFPureThresh}
{\sc S.~Takagi and K.-i. Watanabe}: \emph{On {F}-pure thresholds}, J. Algebra
  \textbf{282} (2004), no.~1, 278--297. {\sf\scriptsize MR2097584
  (2006a:13010)}

\bibitem[Tuc]{TuckerFSignatureExists}
{\sc K.~Tucker}: \emph{{$F$}-signature exists}, Inventiones Mathematicae,
  1--23, 10.1007/s00222-012-0389-0.

\bibitem[Vra08]{VraciuNewatTightClosure}
{\sc A.~Vraciu}: \emph{A new version of {${\mathfrak{a}}$}-tight closure},
  Nagoya Math. J. \textbf{192} (2008), 1--25. {\sf\scriptsize 2477609
  (2009k:13007)}

\bibitem[Yao06]{YaoObservationsAboutTheFSignature}
{\sc Y.~Yao}: \emph{Observations on the {$F$}-signature of local rings of
  characteristic {$p$}}, J. Algebra \textbf{299} (2006), no.~1, 198--218.
  {\sf\scriptsize MR2225772 (2007k:13007)}

\end{thebibliography}

\end{document}